\documentclass[12pt,a4paper,reqno]{amsart}
\usepackage[hmargin=2.5cm,bmargin=2.5cm,tmargin=3cm]{geometry}
\usepackage{enumerate}
\usepackage{amsmath,amsthm,amssymb,amsfonts,latexsym}
\usepackage[gen]{eurosym}
\usepackage[english]{babel}

\usepackage{latexsym}

\usepackage{tikz}					

\newcommand{\Id}{\operatorname{Id}}

\newcommand{\RR}{\mathbb{R}}
\newcommand{\ZZ}{\mathbb{Z}}
\newcommand{\TT}{\mathbb{T}}

\newcommand{\g}{\gamma}

\newtheorem{theorem}{Theorem}
\newtheorem{lemma}[theorem]{Lemma}
\newtheorem{definition}[theorem]{Definition}
\newtheorem{proposition}[theorem]{Proposition}

\newtheorem{remark}[theorem]{Remark}

\newcommand{\be}{\begin{equation}}
\newcommand{\ee}{\end{equation}}
\newcommand{\bea}{\begin{eqnarray*}}
	\newcommand{\eea}{\end{eqnarray*}}
\newcommand{\beq}{\begin{eqnarray}}
\newcommand{\eeq}{\end{eqnarray}}

\newcommand{\pmat}[1]{\begin{pmatrix}#1\end{pmatrix}}

\usepackage{todonotes}

\title[Robustness of flat bands]{Robustness of flat bands on the perturbed Kagome and the perturbed Super-Kagome lattice} 

\subjclass[2010]{}


\author[J.~Kerner]{Joachim Kerner}
\author[M.~T\"aufer]{Matthias T\"aufer}
\author[J.~Wintermayr]{Jens Wintermayr}

\address{Joachim Kerner, Lehrgebiet Analysis, Fakult\"at Mathematik und Informatik, Fern\-Universit\"at in Hagen, D-58084 Hagen, Germany}
\email{joachim.kerner@fernuni-hagen.de}

\address{Matthias T\"aufer, Lehrgebiet Analysis, Fakult\"at Mathematik und Informatik, Fern\-Universit\"at in Hagen, D-58084 Hagen, Germany}
\email{matthias.taeufer@fernuni-hagen.de}

\address{Jens Wintermayr, Bergische Universität Wuppertal,
	Fakultät für Mathematik und Naturwissenschaften, 42119 Wuppertal, Germany}
\email{wintermayr@uni-wuppertal.de}

\date{\today}

\thanks{
}

\begin{document}

\begin{abstract} We study spectral properties of perturbed discrete Laplacians on two-dimen\-sional Archimedean tilings. 
The perturbation manifests itself in the introduction of non-trivial edge weights.  
We focus on the two lattices on which the unperturbed Laplacian exhibits flat bands, namely the  $(3.6)^2$ Kagome lattice and the $(3.12)^2$ ``Super-Kagome'' lattice. We characterize all possible choices for edge weights which lead to flat bands. 
Furthermore, we discuss spectral consequences such as the emergence of new band gaps. 
Among our main findings is that flat bands are robust under physically reasonable assumptions on the perturbation and we completely describe the perturbation-spectrum phase diagram.
The two flat bands in the Super-Kagome lattice are shown to even exhibit an ``all-or-nothing'' phenomenon in the sense that there is no perturbation which can destroy only one flat band while preserving the other.
\end{abstract}

\maketitle

\section{Introduction}
    \label{sec:intro}
This paper is about discrete Schr\"odinger operators on Archimedean tilings, a class of periodic two-dimensional lattices that were already investigated by Johannes Kepler in $1619$~\cite{Kepler-1619}. 
They are natural candidates for the geometry of two-dimensional nanomaterials and due to advances in this field, most prominently represented by graphene, they have become increasingly a focus of attention.

Much work has been devoted to understanding physical properties of such (new) materials~\cite{SuperKagomeConduct,TWGK,Crasto_de_Lima}. 
Most importantly, it can be expected that the underlying geometry, that is the particular lattice, is a key feature determining physical properties of the system. 
In fact, in particular in the mathematical physics literature, investigations of the connection between the geometry (or topology) of a system and the spectral properties of the associated Hamiltonian have become ubiquitous. 
Classical examples in this context are so-called quantum waveguides \cite{ExnerK_2015,Exner_2020,PEThreeDim} as well as quantum graphs~\cite{BerKuch, BaradaranE-22}; see also~\cite{KuchmentP-07} for a relatively recent reference relevant in our context.

A closely related research direction is superconductivity: the existence of a boundary leads to boundary states in a superconductor with a higher critical temperature than the one of the bulk~\cite{SamoilenkaB_20,PhysRevB.103.224516,HSR}. In this spirit, it seems very promising to also study the interplay of geometry and many-particle phenomena on Archimedean tilings. 
Yet another related investigation can be found in \cite{JulkuBEC,SuperKagomeConduct} where another important quantum phenomenon, namely Bose-Einstein condensation, is examined. It turns out that so-called \textit{flat bands}, that are infinitely degenerate eigenvalues of the Hamiltonian, play an important role in understanding such many-particle effects, and for other physical phenomena~\cite{KollaFSH-19}. 
One of the central motivations for this paper is to study robustness of flat bands under certain natural perturbations. 

Two Archimedean tilings, the $(3.6)^2$ \emph{Kagome lattice} and the $(3.12^2)$ tiling
\footnote{We explain the notation for the lattices in Section~\ref{sec:Archimedean_tilings}.}, which we shall dub \emph{Super-Kagome lattice} for reasons that will become clear over the course of the article, stand out: they are the only Archimedean lattices on which the discrete, unweighted Laplacian has flat bands.
In particular the Kagome lattice is a prominent model in physics that has recently enjoyed increasing interest~\cite{PhysRevB.98.235109,KagomeWB,ThesisDias}. 
From a mathematical point of view, our paper is motivated by~\cite{TaeuferPeyerimhoff} where flat bands for the discrete, unweighted Laplacian on Archimedean tilings have been studied in great detail, in combination with an explicit calculation of the integrated density of states. 

A priory, the flat-band phenomena on the Kagome and Super-Kagome lattice seem very sensitive to perturbations: if one replaces the adjacency matrix or the Laplacian by a variant with periodically chosen edge weights, one will generically destroy flat bands.
However, the results of this paper suggest that, if one looks at proper, meaningful variants of the discrete Laplacian which respect certain, natural symmetries of the tiling (we call them \emph{monomeric} Laplacians in Definition~\ref{DefinitionSymmetry}), then flat bands will persist.
Since monomericity is a physically justifiable assumption, this makes a strong case that flat bands are a robust phenomenon, caused by the geometry of the lattice alone and specific to these two lattices, see Theorems~\ref{thm:FlatBandsKagome}, and~\ref{thm:FlatBandsSuperKagome}.

Other questions of interest on periodic graphs concern existence, persistence and estimates on the width of spectral bands and the gaps between them~\cite{KorotyaevS-19,KorotyaevS-19,Mohar-89}.
We will completely identify the spectra as a function of the perturbation in these cases, see Theorems~\ref{thm:BandGapKagome}, and~\ref{thm:SUPERKagomeGaps} as well as Figures~\ref{fig:Kagome}, and~\ref{fig:Super_Kagome}.
This provides an exhaustive description of all nanomaterials based on Archimedean tilings on which discrete Laplacians can exhibit flat bands.

Our paper is organized as follows: 
Sections~\ref{sec:Archimedean_tilings}, and~\ref{sec:Laplacian} are of introductory nature, introducing the notion of and arguing for the relevance of Archimedean tilings, and defining a proper notion of a discrete Laplace operator with non-uniform edge weights.
Section~\ref{sec:Laplacian} also introduces the notion of flat bands and argues why it suffices to restrict our attention to the $(3.6)^2$ Kagome and the $(3.12^2)$ Super-Kagome lattice.
Sections~\ref{sec:Kagome}, and~\ref{sec:Super_Kagome} contain our main results on the Kagome and Super Kagome lattice, respectively.
The contributions of this paper are:
\begin{enumerate}[(i)]
 \item 
 We identify the Kagome and Super-Kagome lattice as the only Archimedean lattices on which a natural class of periodic, weighted Laplacians can have flat bands (Proposition~\ref{prop:no_surprises}).
 \item
 We describe all periodic edge weights which lead to the maximal possible number of bands on the Kagome and Super-Kagome lattice, and prove that this is equivalent to so-called \textbf{monomericity} of the edge weights (Theorems~\ref{thm:FlatBandsKagome} and~\ref{thm:FlatBandsSuperKagome}).
 \item
 We completely describe the spectrum in the monomeric Kagome and Super-Kagome lattice (Theorems~\ref{thm:BandGapKagome} and~\ref{thm:SUPERKagomeGaps}).
 In particular, the monomeric Super-Kagome lattice has a surprisingly rich spectrum-perturbation phase diagram (Figure~\ref{fig:Super_Kagome}) which might bear relevance for various applications.
 \item
 In the Super-Kagome lattice, under a weaker condition than monomericity, namely \textbf{constant vertex weight}, we explicitely describe all remaining ``spurious'' edge weights which have only one flat band.
 We describe the topology of this set within the parameter space and show in particular that it is disconnected from the monomeric two-band set (Theorem~\ref{thm:one_flat_band_SK}).
\end{enumerate}

\section{Archimedean tilings}
    \label{sec:Archimedean_tilings}

\emph{Archimedean}, \emph{Keplerian} or \emph{regular} tilings are edge-to-edge tesselations of the Euclidean plane by regular convex polygons such that every vertex is surrounded by the same pattern of adjacent polygons.
We will adopt the notation of~\cite{GruenbaumS-89} and use the (counterclockwise) order of polygons arranged around a vertex as a symbol for a tiling (this is unique up to cyclic permutations), see Figure~\ref{fig:11_tilings} for the $(3.6)^2$ Kagome lattice and the $(3.12^2)$ Super-Kagome lattice which will be investigated in this paper.
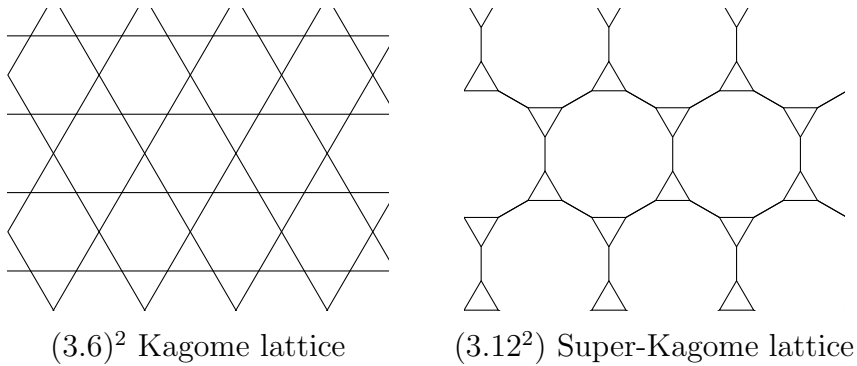
\begin{figure}[ht]
 \begin{tikzpicture}[scale = 2]

\begin{scope}[yshift = -3.5cm]
\draw (1.25,-.25) node {$(3.6)^2$ Kagome lattice};
 \clip (0,0) rectangle (2.5,2);
\begin{scope}[scale = .3]
  \foreach \i in {0,...,6}
    {
    \foreach \j in {0,...,3}
        {
        \pgfmathsetmacro{\ii}{2*\i}
        \pgfmathsetmacro{\jj}{4*\j*0.86602540378}
        \begin{scope}[xshift = \ii cm, yshift = \jj cm]
                \draw (0:1) -- (60:1) -- (120:1) -- (180:1) -- (240:1) -- (300:1) -- (360:1);
                
            \begin{scope}[xshift = 1cm, yshift = 2*0.86602540378cm]
                \draw (0:1) -- (60:1) -- (120:1) -- (180:1) -- (240:1) -- (300:1) -- (360:1);
            \end{scope}
        \end{scope}
        }
    }
\end{scope}
\end{scope}    
    
\begin{scope}[xshift = 3cm, yshift = -3.5cm]
\draw (1.25,-.25) node {$(3.12^2)$ Super-Kagome lattice};
 \clip (0,0) rectangle (2.5,2);
\begin{scope}[scale = .225] 
   \foreach \i in {0,...,6}
    {
    \foreach \j in {0,...,3}
        {
        \pgfmathsetmacro{\ii}{\i * (2*0.86602540378 + 2)}
        \pgfmathsetmacro{\jj}{\j * (4 * 0.86602540378 + 3)}
        \begin{scope}[xshift = \ii cm, yshift = \jj cm]
        \draw (0,0) -- (1,0) -- (.5,0.86602540378) -- (0,0);
        \draw (.5,0.86602540378) -- (.5,1.86602540378);
        \draw (.5,1.86602540378) -- (1,1.86602540378 + 0.86602540378) -- (0,1.86602540378 + 0.86602540378) -- (.5,1.86602540378);
            \begin{scope}
            \draw (210:1) -- (0,0);
            \end{scope}
            \begin{scope}[xshift = 1cm]
            \draw (330:1) -- (0,0);
            \end{scope}
            \pgfmathsetmacro{\l}{2*0.86602540378 + 1}
            \begin{scope}[yshift = \l cm]
            \draw (150:1) -- (0,0);
            \end{scope}
            \begin{scope}[xshift=1cm, yshift = \l cm]
            \draw (30:1) -- (0,0);
            \end{scope}
        \end{scope}
        
        \pgfmathsetmacro{\iii}{\ii + (0.86602540378 + 1)}
        \pgfmathsetmacro{\jjj}{\jj + (2 * 0.86602540378 + 1.5)}
        
        \begin{scope}[xshift = \iii cm, yshift = \jjj cm]
        \draw (0,0) -- (1,0) -- (.5,0.86602540378) -- (0,0);
        \draw (.5,0.86602540378) -- (.5,1.86602540378);
        \draw (.5,1.86602540378) -- (1,1.86602540378 + 0.86602540378) -- (0,1.86602540378 + 0.86602540378) -- (.5,1.86602540378);
            \begin{scope}
            \draw (210:1) -- (0,0);
            \end{scope}
            \begin{scope}[xshift = 1cm]
            \draw (330:1) -- (0,0);
            \end{scope}
            \pgfmathsetmacro{\l}{2*0.86602540378 + 1}
            \begin{scope}[yshift = \l cm]
            \draw (150:1) -- (0,0);
            \end{scope}
            \begin{scope}[xshift=1cm, yshift = \l cm]
            \draw (30:1) -- (0,0);
            \end{scope}
        \end{scope}
        }
    }
\end{scope}
\end{scope}

 \end{tikzpicture}

\caption{The two Archimedean tilings primarily investigated in this article.}
\label{fig:11_tilings}
\end{figure}

The first systematic investigation from $1619$ is due to Kepler who identified all $11$ such tilings~\cite{Kepler-1619}\footnote{All $11$ Archimedean tilings are: the $(4^4)$ rectangular tiling, the $(3^6)$ triangular tiling, the $(6^3)$ hexa\-gonal tiling, the $(3.6^2)$ Kagome lattice, the $(3.12^2)$ Super-Kagome lattice, the $(3^3.4^2)$ tiling, the $(4.8^2)$ tiling, the $(3^2.4.3.4)$ tiling, the $(3.4.6.4)$ tiling, the $(4.6.12)$ tiling, and the $(3^4.6)$ tiling.}.
Most importantly, Archimedean tilings provide natural candidates for geometries of two-dimensional nanomaterials
since they form natural, symmetric arrangements of a single buiding block, positioned at every vertex.
And indeed, these lattices can be observed in many naturally occurring materials~\cite{FK1,FK2,Kulkarni_2020}. 

From a physical point of view, two-dimensional materials such as graphene are interesting since they feature so-called \textit{Dirac points} which are related to a specific behaviour of the electronic band structure of the material~\cite{FW,C3CP53257G,HCDirac}.

Also note that there are deep connections between Laplacians on these lattices, percolation, and self-avoiding walks which have also been studied extensively ~\cite{SykesE-64, Kesten-80,Nienhuis-82, Sud99, Veselic-04, Par07, Jac14, JacobsenSG_16}. An important quantity in this context is the so-called  \emph{connective constant}, which is known only in few cases, for example on the hexagonal lattice~\cite{Duminil-CopinS-12}.

\section{Defining a suitable Hamiltonian}
    \label{sec:Laplacian}

Every Archimedean tiling can be regarded as an infinite discrete graph $G=(V,E)$ with (countable) vertex set $V$ and (countable) edge set $E$.
We write $v \sim w$ if the vertices $v$ and $w$ are joined by an edge and denote by 
\[
\lvert v \rvert := \# \{ w \in V \colon v \sim w \} 
\]
the vertex degree of $v$ (which in the case of Archimedean lattice graphs is $v$-independent).
Archimedean lattices are $\ZZ^2$-periodic, and there exists a cofinite $\ZZ^2$-action 
\[
    \ZZ^2 \ni \beta \mapsto T_\beta \colon V \to V\ ,
\]
that is a group of graph isomorphisms (intuitively understood as a group of shifts) isomorphic to the group $\ZZ^2$.
Let $Q \subset V$ be a minimal (in particular finite) fundamental domain of this action, i.e. the quotient of $V$ under the equivalence relation generated by the group of isomorphisms $(T_\beta)_{\beta \in \ZZ^2}$.

In the unweighted case, a natural, normalized choice for the Hamiltonian is the discrete Laplacian
\begin{equation}
 \label{eq:unweighted_Laplacian}
 (\Delta f)(v)
 :=
 \frac{1}{\lvert v \rvert}
 \sum_{w \sim v}
 \left(f(v) - f(w) \right)
 =
 f(v)
 -
 \frac{1}{\lvert v \rvert}
 \sum_{w \sim v}
 f(w)\ ,
\end{equation}
as used for instance in~\cite{TaeuferPeyerimhoff}.
It can be written as $\Delta f = \Id - \frac{1}{\lvert v \rvert} \Pi$ where $\Pi$ is the adjacency matrix, that is $\Pi(v,w) = 1$ if $v \sim w$ and $0$ else.
The following is standard:
\begin{lemma}
The unweighted, normalized Laplacian~\eqref{eq:unweighted_Laplacian} with a uniformly bounded vertex degree boasts the following properties:
\begin{enumerate}[(i)]
 \item 
 All restrictions of $\Delta$ to finitely many vertices are real-symmetric $M$-matrices, that is, their off-diagonal elements are non-positive¸ and all their eigenvalues are non-negative.
 \item
 The infimum of the spectrum of $\Delta$ is $0$.
 \item
 All rows and columns of $\Delta$ sum to zero.
\end{enumerate}
Furthermore, the spectrum is always contained in the interval $[0,2]$.
\end{lemma}

Introducing non-trivial edge weights, we would like to keep a form of the Laplacian that preserves properties (i) to (iii). 
A natural candidate, similar to formula~(2.11) in \cite{KorotyaevS-21}, is
\begin{equation}\label{eq:Laplacian_KS}
(\Delta_{\gamma}f)(v):=\frac{1}{\sqrt{\mu(v)}}\sum_{w\sim v}\gamma_{vw}\left(\frac{f(v)}{\sqrt{\mu(v)}}-\frac{f(w)}{\sqrt{\mu(w)}}\right)
\end{equation}
where the \emph{edge weights} $\gamma_{vw}=\gamma_{wv} > 0$ and \emph{vertex weights} $\mu(v)$ satisfy the relation
\begin{equation}
    \label{eq:vertex_and_edge_weights}
 \sum_{w \sim v} \gamma_{vw} = \mu_v
 \quad
 \text{for every $v \in V$}.
\end{equation}
As long as the vertex weights $\mu(v)$ (and thus also the $\gamma_{vw}$) are uniformly bounded, this will lead to an operator with properties (i) to (iii) and spectrum contained in $[0,2]$.

\begin{remark}
 In the literature, one often finds the definition
 \[
  (\Delta_\gamma f) (v)
  =
  \frac{1}{\mu(v)} 
  \sum_{w \sim v}
  \gamma_{vw}
  \left(
  f(v) - f(w)
  \right)
 \]
 as a normalized, discrete Laplacian.
 Note that, whenever $\mu(v) \neq \mu(w)$ for some $v \sim w$, then this will not lead to a self-adjoint operator, but it can be made self-adjoint on a suitably weighted $\ell^2(V)$-space, cf.~\cite{KellerLenzGraphs}.
 If all $\mu(v)$ are the same, then this definition coincides with~\eqref{eq:Laplacian_KS}, and can be simplified to
 \begin{equation}\label{eq:laplacian}
(\Delta_{\gamma}f)(v)=f(v)-\frac{1}{\mu}\sum_{w \sim v} \gamma_{wv}f(w)\ .
\end{equation}
\end{remark}

Now, one can prescribe various degrees of the symmetry of the underlying Archimedean lattice to be respected by the Laplacian:
\begin{definition}\label{DefinitionSymmetry}
 Consider an Archimedean tiling $(V,E)$ with periodic edge weights $\gamma_{vw} = \gamma_{wv} > 0$, that is $\gamma_{vw} = \gamma_{T_\beta v T_{\beta} w}$ for all $v,w \in V$ and $\beta \in \ZZ^2$, and corresponding vertex weights $\mu(v) = \sum_{w \sim v} \gamma_{vw}$. Define the Laplacian $\Delta_\gamma$ as in~\eqref{eq:Laplacian_KS}.
 Then, we say that the Archimedean tiling with Laplacian $\Delta_\gamma$
 \begin{enumerate}[(1)]
  \item
  has \textbf{constant vertex weight}, if there is $\mu > 0$ such that $\mu(v) = \mu$ for all $v \in V$.
  \item
  is \textbf{monomeric} if for all vertices $v \in V$ the list of edge weights, arranged cyclically around $v$, coincides (up to cyclic permutations).
  \end{enumerate} 
\end{definition}
Clearly, (2) is stronger than (1). However, in either case, the Laplacian reduces to~\eqref{eq:laplacian}.

The term ``monomeric'' is inspired by the fact that the associated operators can be interpreted as describing properties of nanomaterials formed from one type of monomeric building block, positioned at every vertex of an Archimedean tiling.
Clearly, monomeric Laplacians on Archimedean lattices have constant vertex weights, but the converse is not true in general.
However, we will see in Theorems~\ref{thm:FlatBandsKagome} and~\ref{thm:FlatBandsSuperKagome} that on the Kagome and Super-Kagome lattice, the validity of the converse implication is equivalent to existence (or persistence) of all flat bands. 
Also, monomericity seems a physically reasonable assumption for nanomaterials, which suggests that the emergence of flat bands, while a priori very sensitive to perturbations of coefficients in the operator, might nevertheless be robust within the class of physically relevant operators.

Next, let $\TT^2 = \RR^2 / \ZZ^2$ be the flat torus and define for every $\theta \in \TT^2$ the $|Q|$-dimensional Hilbert space
\begin{equation*}
\ell^2(V)_{\theta}:=\{f:V \rightarrow \mathbb{C}\ |\ f(T_{\beta}v)=e^{i\langle \theta,\beta \rangle}f(v)  \}\ 
\end{equation*}
with inner product 
\begin{equation*}
\langle f,g\rangle_{\theta}:=\sum_{v \in Q}f(v)\overline{g(v)}\ .
\end{equation*}
Given the Laplacian~\eqref{eq:laplacian} on $\ell^2(V)$ with properties described in Definition~\ref{DefinitionSymmetry}, we define on $\ell^2(V)_{\theta}$ the operator 
\begin{equation}\label{eq:Laplacian_2}
(\Delta^{\theta}_{\gamma}f)(v):=f(v)-\frac{1}{\mu}\sum_{w \sim v} \gamma_{wv}f(w)\ .
\end{equation}
Clearly, \eqref{eq:Laplacian_2} can be represented as a $\lvert Q \rvert$-dimensional Hermitian matrix.
Due to Floquet theory, we have
\[
 \sigma (\Delta_\gamma)
 =
 \bigcup_{\theta \in \TT^2}
 \sigma(\Delta_\gamma^\theta)\ ,
\]
and the following statement holds.
\begin{proposition}[See~\cite{TaeuferPeyerimhoff} and references therein]\label{PropositionFlatBand}
 Let $E \in \RR$. 
 Then, the following are equivalent:
 \begin{enumerate}[(i)]
  \item 
  $E \in \sigma( \Delta_\gamma^\theta)$ for all $\theta \in \TT^2$.
  \item
  $E \in \sigma( \Delta_\gamma^\theta)$ for a positive measure subset of $\theta \in \TT^2$.
  \item
  There is an infinite orthonormal family eigenfunctions of $\Delta_\gamma$ to the eigenvalue $E$.
  Each of them can be chosen to be supported on a finite number of vertices.
 \end{enumerate}
 If any of (i) to (iii) is satisfied, we say that $\Delta_\gamma$ has a \textbf{flat band} (at energy $E$).
\end{proposition}
Note that, in the $\ell^\infty(V)$ setting instead of the $\ell^2(V)$ setting, such infinitely degenerate eigenvalues are also referred to as ``black hole eigenvalues'' in~\cite{vonBelowL-09}.
Also, the existence of flat bands can be interpreted as a breakdown of the unique continuation principle~\cite{PeyerimhoffTV-17}.

In the Hilbert space $\ell^2(V)$ setting, is known that for constant edge weights, the discrete Laplacian has flat bands only on two of the $11$ Archimedean lattices, namely the $(3.6)^2$ Kagome lattice and the $(3.12^2)$ Super-Kagome lattice~\cite{TaeuferPeyerimhoff}.
Before turning to perturbed versions of those two lattices, one should verify that there won't be any surprises on the other lattices:
\begin{proposition}
 \label{prop:no_surprises}
 On the Archimedean lattices $(4^4)$, $(3^6)$, $(6^3)$, $(3^3.4^2)$, $(4.8^2)$, $(3^2.4.3.4)$, $(3.4.6.4)$, $(4.6.12)$, $(3^4.6)$, there is no choice of periodic (with respect to the fundamental cell on the lattice) edge weights $\gamma_{vw} = \gamma_{wv} > 0$ which will make the weighted adjacency matrix
 \[
  \Pi_\gamma(v,w) 
  =
  \begin{cases}
   \gamma_{vw} & \text{if $v \sim w$},\\
   0 & \text{else}
  \end{cases}
 \]
 have a flat band.
 Consequently, also the Laplacian with constant or monomeric edge weights has no flat bands on these lattices.
\end{proposition}

Proposition~\ref{prop:no_surprises} is proved by a series of straightforward but somewhat lengthy calculations in which one calculates the associated characteristic polynomials, and shows that there are no $\theta$-independent roots, employing Proposition~\ref{PropositionFlatBand} (this should be compared to the proofs of Theorems~\ref{thm:FlatBandsKagome} and~\ref{thm:FlatBandsSuperKagome} below).
We omit them here for the sake of conciseness.
In any case, Proposition~\ref{prop:no_surprises} justifies to restrict our attention to the (perturbed) Kagome and Super-Kagome lattices from now on.

\section{The perturbed Kagome lattice}
    \label{sec:Kagome}
In this section we discuss the Kagome lattice with non-uniform (periodic) edge weights. The elementary cell of the Kagome lattice contains three vertices and six edges (one can think of the edges as arranged around a hexagon).
\begin{figure}[ht]
\begin{tikzpicture}[scale = 1.5]

\begin{scope}[xshift = 5cm]

   \draw[fill = black] (60:1) circle (2pt);
   \draw[fill = black] (120:1) circle (2pt);
   \draw[fill = black] (180:1) circle (2pt);
   
   \begin{scope}[xshift = 2cm]
    \draw[fill = white] (120:1) circle (2pt);
    \draw[fill = white] (180:1) circle (2pt);
   \end{scope}
   \begin{scope}[xshift = -1cm , yshift = 2 * 0.86602540378 * 1 cm]
    \draw[fill = white] (0:1) circle (2pt);
    \draw[fill = white] (240:1) circle (2pt);
   \end{scope}
   \draw[fill = white] (240:1) circle (2pt);
   \begin{scope}[xshift = -2cm]
    \draw[fill = white] (300:1) circle (2pt);
   \end{scope}

   \draw[fill = white] (300:1) circle (2pt);
   \draw[thick, dashed] (247:.94) -- (293:.94);
   \draw[thick, dashed] (307:.94) -- (-7:.94);
   
    \draw[thick] (7:.94) -- (53:.94);
    \draw[thick] (67:.94) -- (113:.94);
    \draw[thick] (127:.94) -- (173:.94);
    \draw[thick] (187:.94) -- (233:.94);

\begin{scope}[xshift = 1cm, yshift = 2 * 0.86602540378 * 1 cm]
    \draw[thick] (187:.94) -- (233:.94);
    \draw[thick] (247:.94) -- (293:.94);
\end{scope}
\begin{scope}[xshift = -1cm, yshift = 2 * 0.86602540378 * 1 cm]
    \draw[thick] (247:.94) -- (293:.94);
    \draw[thick] (307:.94) -- (353:.94);
\end{scope}
\begin{scope}[xshift = -2cm]
    \draw[thick] (7:.94) -- (53:.94);
    \draw[thick] (307:.94) -- (353:.94);
\end{scope}

   \draw (210:.65) node {\color{red}$\gamma_1$};
   \draw (150:.65) node {\color{red}$\gamma_2$};
   \draw (90:.65) node {\color{red}$\gamma_3$};
   \draw (30:.65) node {\color{red}$\gamma_4$};
   \draw (-30:.65) node {\color{red}$\gamma_5$};
   \draw (-90:.65) node {\color{red}$\gamma_6$};
   
   \begin{scope}[xshift = -2cm]
    \draw (30:.65) node {\color{red}$\gamma_4$};
    \draw (-30:.65) node {\color{red}$\gamma_5$};
   \end{scope}   
   
   \begin{scope}[xshift = -1cm]
    \draw (90:1) node {\color{red}$\gamma_6$};
   \end{scope}   

   \begin{scope}[xshift = -1cm , yshift = 2 * 0.86602540378 * 1 cm]
    \draw (-30:.65) node {\color{red}$\gamma_5$};
   \end{scope}
   
   \begin{scope}[xshift = 1cm , yshift = 2 * 0.86602540378 * 1 cm]
    \draw (210:.65) node {\color{red}$\gamma_1$};
   \end{scope}
   
   \begin{scope}[xshift = 1cm]
    \draw (90:1) node {\color{red}$\gamma_6$};
   \end{scope}
   
   \draw (60:1.2) node {$v_1$};
   \draw (120:1.2) node {$v_2$};
   \draw (180:1.2) node {$v_3$};

   \begin{scope}[xshift = 2cm]
    \draw (90:1) node {$v_2 + \omega_1$};
    \draw (180:.4) node {$v_3 + \omega_1$};
   \end{scope}
   \begin{scope}[xshift = -1cm , yshift = 2 * 0.86602540378 * 1 cm]
    \draw (10:1) node {$v_3 + \omega_2$};
   \end{scope}
   \draw (250:1.25) node {$v_1 - \omega_2$};
   \begin{scope}[xshift = -2cm]
    \draw (90:1) node {$v_1 - \omega_1$};
    \draw (270:1) node {$v_2 - \omega_2$};
   \end{scope}
   
\end{scope}

\end{tikzpicture}
 \caption{Fundamental domain of the Kagome lattice with edge weights.
 In the monomeric case, all edge weights around downwards pointing triangles are $\g_2 = \g_4 = \g_6 =: \alpha$ and all edge weights on upwards pointing triangles are $\g_1 = \g_3 = \g_5 =: \beta$, where $2 \alpha + 2 \beta = \mu$.}
 \label{fig:kagome}
\end{figure}
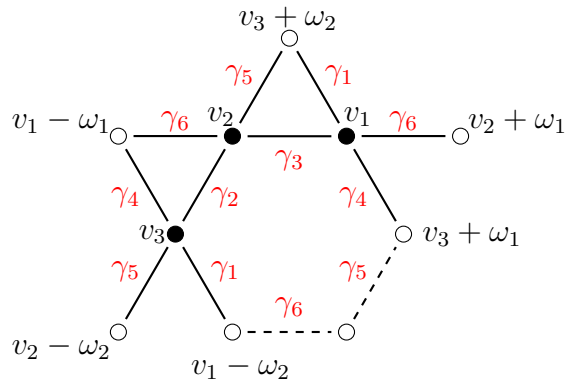
A priori, periodicity allows for six edge weights $\gamma_1,...,\gamma_6 > 0$, and the Floquet Laplacian $\Delta_\gamma^\theta$ can be written as the Hermitian matrix
\begin{equation}
    \label{MatrixHamKagome}
    \Delta^{\theta}_{\gamma}
    =
    \Id-\frac{1}{\mu}\begin{pmatrix} 0 & \gamma_3+w\gamma_6 & w\gamma_4+z\gamma_1 \\ \gamma_3+\overline{w}\gamma_6 & 0 & \gamma_2+z\gamma_5 \\ \overline{w}\gamma_4+\overline{z}\gamma_1 & \gamma_2+\overline{z}\gamma_5 & 0
\end{pmatrix}\ ,
\end{equation}
where $w:=e^{i\theta_1}$ and $z:=e^{i\theta_2}$. 
We denote the three real eigenvalues of $\Delta^{\theta}_{\gamma}$ by $\lambda_1(\theta,\gamma) \leq \lambda_2(\theta,\gamma) \leq  \lambda_3(\theta,\gamma)$.

Note that the six degrees of freedom are to be further reduced, depending on the following symmetry conditions:
\begin{itemize}
    \item
    If we merely assume a \emph{constant vertex weight} $\mu > 0$, then identity~\eqref{eq:vertex_and_edge_weights} will impose the three additional linearly independent conditions
\begin{equation}\label{VertexCond}\begin{split}
\gamma_1+\gamma_4&=\gamma_2+\gamma_5\ , \\
\gamma_3+\gamma_6&=\gamma_2+\gamma_5\ ,\\
\gamma_1+\gamma_3+\gamma_4+\gamma_6&= \mu\ ,
\end{split}
\end{equation}
and we end up with three degrees of freedom.
    \item
    If we also assume \emph{monomericity}, then it is easy to see that the only choice is the \emph{breathing Kagome lattice}, cf.~\cite{HerreraKPGSSB-22}, with an edge weight $\alpha > 0$ on all edges belonging to upwards pointing triangles and edge weight $\beta > 0$ on all edges belonging to downwards pointing triangles, where $2 (\alpha + \beta) = \mu$.
    After fixing the vertex weight $\mu$, this amounts to only one degree of freedom.
\end{itemize}
\subsection{Flat bands in the perturbed Kagome lattice}\label{SectionFlatKagome}
%
%
\begin{theorem}\label{thm:FlatBandsKagome} 
Consider the perturbed Kagome lattice with Laplacian~\eqref{eq:laplacian}, fixed vertex weight $\mu > 0$ and periodic edge weights $\gamma_1,...,\gamma_6 > 0$, satisfying the condition~\eqref{eq:vertex_and_edge_weights} on vertex and edge weights. 
Then, the following are equivalent:
\begin{enumerate}[(i)]
 \item 
 There exists a flat band.
 \item
 The vertex weights are monomeric. More explicitly, there are $\alpha, \beta > 0$ with $2 (\alpha + \beta) = \mu$ such that
 \begin{equation*}\begin{split}
 \gamma_{2}&=\gamma_{4}=\gamma_{6}:=\alpha, \\
 \gamma_{1}&=\gamma_{3}=\gamma_{5}:=\beta.
 \end{split}
 \end{equation*}
\end{enumerate}
%
\end{theorem}

The rest of this subsection is devoted to the proof of Theorem~\ref{thm:FlatBandsKagome}. 
We start with identifying flat bands using the weighted adjacency matrix
\begin{equation}\Pi^{\theta}_{\gamma}:= \begin{pmatrix} 0 & \gamma_3+w\gamma_6 & w\gamma_4+z\gamma_1 \\ \gamma_3+\overline{w}\gamma_6 & 0 & \gamma_2+z\gamma_5 \\ \overline{w}\gamma_4+\overline{z}\gamma_1 & \gamma_2+\overline{z}\gamma_5 & 0
\end{pmatrix}
\end{equation}
which is spectrally equivalent to $\Delta_\gamma^\theta$ up to scaling and shifting via the relation
\[
 \Delta_\gamma^\theta
 =
 \Id
 -
 \frac{1}{\mu}
 \Pi_\gamma^\theta.
\]
In order to find flat bands, we will identify conditions for $\theta$-independent eigenvalues of $\Pi_\gamma^\theta$ and therefore calculate
\begin{equation*}\begin{split}
\det(\lambda\Id-\Pi^{\theta}_{\gamma})&=-\lambda^3+\lambda(|A|^2+|B|^2+|C|^2)+2 \Re (A\overline{B}C)
\end{split}
\end{equation*}
where $A:=\gamma_3+w\gamma_6$, $B:=w\gamma_4+z\gamma_1$ and $C:=\gamma_2+z\gamma_5$. Rearranging the terms yields
\begin{equation*}\begin{split}
    \det(\lambda\Id-\Pi^{\theta}_{\gamma})=
    &(w + \overline w) (\lambda\gamma_6 \gamma_3+\gamma_3\gamma_2\gamma_4+\gamma_6\gamma_5\gamma_1) \\
    +
    &(z + \overline z) (\lambda \gamma_5 \gamma_2+\gamma_6\gamma_5\gamma_4+\gamma_1\gamma_3\gamma_2)\\
    +
    &(w \overline z + z\overline{w}) (\lambda \gamma_1\gamma_4+\gamma_3 \gamma_5\gamma_4+\gamma_6\gamma_2\gamma_1) \\
    +
    &(-\lambda^3+\lambda(\gamma^2_1+...+\gamma^2_6)+2(\gamma_4\gamma_6\gamma_2+\gamma_3\gamma_5\gamma_1))\ .
\end{split}
\end{equation*}
The prefactors
\[
    w + \overline{w} = 2 \cos \theta_1\ ,
    \quad
    z + \overline{z} = 2 \cos \theta_2\ ,
    \quad
    \text{and}
    \quad
    w \overline{z} + z \overline{w} = 2 \cos (\theta_1 - \theta_2)\ ,
\]
are linearly independent as measurable functions of $\theta$ on $\TT^2$. Consequently, since all $\gamma_i$ are positive, $\theta$-independent eigenvalues exist if and only if the $w$ and $z$-independent terms in every line are zero.
This is only possible for negative $\lambda$, which (possibly after scaling  the $\gamma_i$ and $\mu$ for the moment) can be assumed to equal $-1$. 
Therefore, we obtain the conditions
\begin{equation}\label{CondI}\begin{split}
\gamma_{3}\gamma_{6}&=\gamma_{2}\gamma_{3}\gamma_{4}+\gamma_{1}\gamma_{5}\gamma_{6}\ , \\
\gamma_{2}\gamma_{5}&=\gamma_{4}\gamma_{5}\gamma_{6}+\gamma_{1}\gamma_{2}\gamma_{3}\ , \\
\gamma_{1}\gamma_{4}&=\gamma_{3}\gamma_{4}\gamma_{5}+\gamma_{1}\gamma_{2}\gamma_{6}\ , 
\end{split}
\end{equation}
and
\begin{equation}\label{CondII}\begin{split}
1-(\gamma^2_{1}+\dots+\gamma^2_{6})+2\left(\gamma_{2}\gamma_{4}\gamma_{6}+\gamma_{1}\gamma_{3}\gamma_{5}\right)=0\ .
\end{split}
\end{equation}
\begin{lemma}\label{LemmaEW} The only positive solutions (meaning all $\gamma_i$ are non-zero) of \eqref{VertexCond}, \eqref{CondI}, \eqref{CondII} are
	\begin{equation}   
	\label{eq:only_positive_solutions} 
    \begin{split}
	\gamma_{2}&=\gamma_{4}=\gamma_{6} = x \\
		\gamma_{1}&=\gamma_{3}=\gamma_{5} = y
	\end{split}
	\end{equation}
	with $x,y \in (0,1)$ and $x+y=1$.
\end{lemma}
\begin{proof} 
By a direct calculation \eqref{eq:only_positive_solutions} solves~\eqref{VertexCond}, \eqref{CondI}, \eqref{CondII}.

Conversely, assume that there are positive solutions $\gamma_1,...,\gamma_6 > 0$. From~\eqref{CondI} we obtain
	\begin{equation*}\begin{split}
	\gamma_3=\frac{\gamma_{1}\gamma_{5}\gamma_{6}}{\gamma_6-\gamma_2\gamma_4}\ , \quad \gamma_1=\frac{\gamma_{3}\gamma_{4}\gamma_{5}}{\gamma_4-\gamma_2\gamma_6}\ , 
		\end{split}
	\end{equation*}
	and this implies $\gamma_6 > \gamma_2 \gamma_4$ and $\gamma_4 > \gamma_2\gamma_6$. Hence, combining both equations yields $\gamma^2_2\gamma_6 < \gamma_6$ which shows that $\gamma_2 < 1$. In the same way one proves $\gamma_i < 1$ for every other $i$.
	
Next, let $\gamma_2+\gamma_5:=\Lambda$. By \eqref{VertexCond} one immediately concludes $\gamma_1+\gamma_4=\gamma_3+\gamma_6=\Lambda$. Now, we add \eqref{CondI} and \eqref{CondII} and rearrange the equations to obtain
\begin{equation*}\begin{split}
\frac{1}{2}\left(\gamma^2_{1}+\dots+\gamma^2_{6}-1\right)+\gamma_{3}\gamma_{6}+\gamma_{2}\gamma_{5}+\gamma_{1}\gamma_{4}=&\gamma_{2}\gamma_{4}\gamma_{6}+\gamma_{1}\gamma_{3}\gamma_{5}\\ &+\gamma_{2}\gamma_{3}\gamma_{4}+\gamma_{1}\gamma_{5}\gamma_{6}+\gamma_{4}\gamma_{5}\gamma_{6}\\
&+\gamma_{1}\gamma_{2}\gamma_{3}+\gamma_{3}\gamma_{4}\gamma_{5}+\gamma_{1}\gamma_{2}\gamma_{6}\ .
\end{split}
\end{equation*}
By repeated factorization, the right hand side simplifies to

\begin{equation}
  \gamma_{1}\gamma_{3}(\gamma_2+\gamma_5)
  +
  \gamma_3\gamma_4(\gamma_2+\gamma_5)
  +
  \gamma_1\gamma_6(\gamma_2+\gamma_5)
  +
  \gamma_4\gamma_6(\gamma_2+\gamma_5)
  =
 \Lambda^3,
\end{equation}
and since for the left hand side one has
\[
 \frac{1}{2}\left(\gamma^2_{1}+\dots+\gamma^2_{6}-1\right)+\gamma_{3}\gamma_{6}+\gamma_{2}\gamma_{5}+\gamma_{1}\gamma_{4}=\frac{3\Lambda^2-1}{2}\ ,
 \]
 we arrive at the polynomial $\Lambda^3-\frac{3\Lambda^2}{2}+\frac{1}{2}=0$
the only positive solution of which is $\Lambda=1$. Finally, adding the first the two equations of \eqref{CondI} yields 
\begin{equation*}\begin{split}
\gamma_3 \gamma_6+\gamma_2 \gamma_5= (\gamma_6\gamma_5+\gamma_2\gamma_3)(\gamma_1+\gamma_4)=\gamma_6\gamma_5+\gamma_2\gamma_3 
\end{split}
\end{equation*}
and this implies $\gamma_5=\gamma_3$. Furthermore, adding the last two equations gives
\begin{equation*}\begin{split}
\gamma_2 \gamma_5+\gamma_1 \gamma_4= (\gamma_4\gamma_5+\gamma_1\gamma_2)(\gamma_3+\gamma_6) =\gamma_4\gamma_5+\gamma_1\gamma_2
\end{split}
\end{equation*}
giving $\gamma_4=\gamma_2$. Conditions~\eqref{VertexCond} hence give $\gamma_1=\gamma_5$ and $\gamma_6=\gamma_2$. This proves the statement.
\end{proof}
We are now in the position to prove Theorem~\ref{thm:FlatBandsKagome}.
\begin{proof}[Proof of Theorem~\ref{thm:FlatBandsKagome}] Comparing $\Pi^{\theta}_\gamma$ with $\Delta^{\theta}_{\gamma}$ we conclude that $\Delta^{\theta}_\gamma$ has a flat band with edge weights $\g_1,...,\g_6$ if and only if there exists $\delta > 0$ such that $\Pi^{\theta}_{\gamma}$ has a flat band for edge weights $\delta \g_1,...,\delta\g_6$. From this observation the statement follows directly taking Lemma~\ref{LemmaEW} into account.

\end{proof}
\subsection{The spectrum and band gaps in the monomeric Kagome lattice}
In the case where the perturbed Kagome lattice has a flat band, we further study the structure of the rest of the spectrum. We reiterate that, due to Theorem~\ref{thm:FlatBandsKagome}, the existence of a flat band is equivalent to the weights being monomeric. 

As shown for instance in \cite{TaeuferPeyerimhoff}, in the case where all edge weights are equal, the two other spectral bands, generated by the two other $\theta$-dependent eigenvalues of $\Delta^{\theta}_{\gamma}$, touch at $E = 3/4$, and the derivative of the integrated density of states at $E = 3/4$ vanishes 
-- an indication that the spectral density at $3/4$ is sufficiently thin for a gap to form under perturbation. And indeed, this is the statement of the next theorem, which also characterises the width of the gap.
\begin{theorem}[Band gaps in the perturbed Kagome lattice]\label{thm:BandGapKagome} 
Consider the perturbed Kago\-me lattice with fixed vertex weight $\mu > 0$, and monomeric edge weights $\alpha, \beta > 0$, satisfying $2 (\alpha + \beta) = \mu$ as characterized in Theorem~\ref{thm:FlatBandsKagome}. 
Then, the spectrum is given by
\[
 I_1
 \cup
 I_2
 :=
 \left[
    0, \frac{3}{4} - \left|\frac{3\alpha}{\mu}-\frac{3}{4}\right|
 \right]
 \bigcup
 \left[
    \frac{3}{4} + \left|\frac{3\alpha}{\mu}-\frac{3}{4}\right|, \frac{3}{2}
 \right].
\]
Furthermore, there is always a flat band at $\frac{3}{2}$.
\end{theorem}
\begin{remark}
 Theorem~\ref{thm:BandGapKagome} states that, as soon as $\alpha \neq \beta$, or alternatively, $\alpha \neq \frac{\mu}{4}$, a spectral gap of width
 \[
  \left|\frac{6\alpha}{\mu}-\frac{3}{2}\right|
  =
  \frac{3}{\mu} \lvert \alpha - \beta \rvert 
 \]
 will form around $\frac{3}{4}$, see also Figure~\ref{fig:Kagome}.
 The flat band at $\frac{3}{2}$ will always be connected to the energy band below it which means that the ``touching'' of the flat band at $\frac{3}{2}$ is protected in the class of monomeric perturbations.
\end{remark}

	\begin{figure}[ht]
    \begin{tikzpicture}[xscale = 4, yscale = 8]

        \fill[red!20] (0,0) -- (.75,.25) -- (0,.5) -- (0,0);
        
        \fill [red!20] (1.5,0) -- (.75,.25) -- (1.5, .5) -- (1.5,0);

        \draw[red, very thick, dashed] (1.5,0) -- (1.5,.5);
        
        \draw (.375,.25) node {$I_1$};
        \draw (1.125,.25) node {$I_2$};

        \draw[thick, ->] (0,0) -- (1.55,0);
        \draw[thick] (0,0) -- (0,.5);
        
        \draw (-.2,.5) node {$\alpha = \frac{\mu}{2}$};
        \draw[thick] (-.05,.5) -- (.05,.5);       
        \draw (-.2,.0) node {$\alpha = 0$};
        \draw (-.2,.25) node {$\alpha = \frac{\mu}{4}$};
        
        \draw (1.5,-.05) node {$\frac{3}{2}$};
        \draw (.75,-.05) node {$\frac{3}{4}$};

        \begin{scope}[xshift = 1.7cm]
         \draw[very thick, red, dashed] (0,0) -- (.25,0);
         \draw[anchor = west] (.25,0) node {Flat band};

         \fill[red!20] (0,.1) rectangle (.25,.2);
         \draw[anchor = west] (.25,.15) node {$\sigma(\Delta_\gamma)$};

        \end{scope}
    \end{tikzpicture}

    \caption{Spectrum of the monomeric $(3^2.6^2)$ Kagome lattice with vertex weight $\mu > 0$ as a function of the parameter $\alpha \in (0, \frac{\mu}{2})$, describing the edge weights on edges adjacent to downwards pointing triangles.}
    \label{fig:Kagome}
	\end{figure}
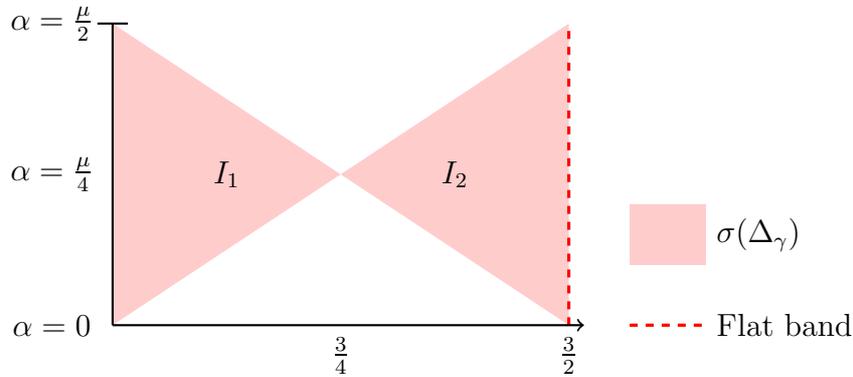

\begin{proof} A calculation shows that the eigenvalues of $\Delta_\gamma^\theta$ with the choice $2( \alpha + \beta)=\mu$ as in Theorem~\ref{thm:FlatBandsKagome} are given by
	\begin{equation*}
	\lambda_{1,2}(\theta,\gamma)=\frac{3}{4}\pm \frac{1}{4}\sqrt{1+8\left(1+(F(\theta)-3)\left(\frac{2\alpha}{\mu}-\frac{4\alpha^2}{\mu^2}\right)\right)}
	\end{equation*}
	and 
	\begin{equation*}
	\lambda_{3}(\theta,\gamma)=\frac{3}{2}
	\end{equation*}
where $F(\theta):=\cos(\theta_1) + \cos(\theta_2) + \cos(\theta_1 - \theta_2)$.
The function $\TT^2 \ni \theta \mapsto F(\theta)$ takes all values in $[-3/2, 3]$, see Lemma~3.1 in~\cite{TaeuferPeyerimhoff}, whence $\lambda_1(\theta, \gamma)$ and $\lambda_2(\theta, \gamma)$ take all values in the intervals
\[
 \left[
    0, \frac{3}{4} -  \left|\frac{3\alpha}{\mu}-\frac{3}{4}\right|
 \right]
 ,
 \quad
 \text{and}
 \quad
 \left[
    \frac{3}{4} + \left|\frac{3\alpha}{\mu}-\frac{3}{4}\right|, \frac{3}{2}
 \right],
 \quad
 \text{respectively}.
 \qedhere
\]
\end{proof}

\section{The perturbed Super-Kagome lattice}
    \label{sec:Super_Kagome}
In this section, we investigate the Archimedean tiling $(3.12^2)$ which we call Super-Kagome lattice.
Its minimal elementary cell contains six vertices and nine edges: three edges on upwards pointing triangles, three edges on downwards pointing triangles, and three edges bordering two dodecagons, see Figure~\ref{fig:super_kagome}.

\begin{figure}[ht]
 \begin{tikzpicture}
  \draw[fill = black] (0,0) circle (2pt);
  \draw (.4,0) node {$v_4$};
  \draw[fill = black] (0,1) circle (2pt);
  \draw (.4,1) node {$v_3$};
  \draw[fill = black] (-60:1) circle (2pt);
  \draw (-75:1.1) node {$v_5$};
  \draw[fill = black] (-120:1) circle (2pt);
  \draw (-105:1.15) node {$v_6$};
  
  \begin{scope}[yshift = 1cm]
  \draw[fill = black] (60:1) circle (2pt);  
  \draw (75:1.1) node {$v_2$};
  \draw[fill = black] (120:1) circle (2pt);   
  \draw (105:1.1) node {$v_1$};
  \end{scope}
  \draw[thick] (0,.2) -- (0,.8);
  \draw[thick] (-60:.8)   -- (-60:.2);
  \draw[thick] (-120:.8)   -- (-120:.2);
  \draw[thick] (-70:.91) -- (-110:.91);
  \begin{scope}[yshift = 1cm]
  \draw[thick] (60:.8)   -- (60:.2);
  \draw[thick] (120:.8)   -- (120:.2);
  \draw[thick] (70:.91) -- (110:.91);
  \end{scope}
  
  \begin{scope}[xshift = .5 cm, yshift = -0.86602540378 * 1 cm]
   \draw [fill = white] (-30:1) circle (2pt);
   \draw (-40:1.2) node {$v_1 - \omega_2$};
   \draw [thick] (-30:.2) -- (-30:.8);
  \end{scope}
  \begin{scope}[xshift = -.5 cm, yshift = -0.86602540378 * 1 cm]
   \draw [fill = white] (-150:1) circle (2pt);
   \draw (-140:1.2) node {$v_2 - \omega_1$};
   \draw [thick] (-150:.2) -- (-150:.8);
  \end{scope}
  \begin{scope}[xshift = .5 cm, yshift = 1.86602540378 * 1 cm]
   \draw [fill = white] (30:1) circle (2pt);
   \draw (40:1.2) node {$v_6 + \omega_1$};
   \draw [thick] (30:.2) -- (30:.8);
  \end{scope}
  \begin{scope}[xshift = -.5 cm, yshift = 1.86602540378 * 1 cm]
   \draw [fill = white] (150:1) circle (2pt);
   \draw [thick] (150:.2) -- (150:.8);
   \draw (140:1.2) node {$v_5 + \omega_2$};
  \end{scope}
  
  
  \begin{scope}[xshift = 5cm]
  \draw[fill = black] (0,0) circle (2pt);
  \draw[fill = black] (0,1) circle (2pt);
  \draw[fill = black] (-60:1) circle (2pt);
  \draw[fill = black] (-120:1) circle (2pt);
  \draw (.25,.5) node {\color{red}$\gamma_7$};
  \draw (-40:.7) node {\color{red}$\gamma_3$};
  \draw (220:.7) node {\color{red}$\gamma_2$};
  \draw (0,-1.1) node {\color{red}$\gamma_1$};
  \begin{scope}[yshift = 1cm]
  \draw[fill = black] (60:1) circle (2pt);  
  \draw (40:.7) node {\color{red}$\gamma_5$};
  \draw[fill = black] (120:1) circle (2pt);   
  \draw (140:.7) node {\color{red}$\gamma_6$};
  \draw (0,1.1) node {\color{red}$\gamma_4$};
  \end{scope}
  \draw[thick] (0,.2) -- (0,.8);
  \draw[thick] (-60:.8)   -- (-60:.2);
  \draw[thick] (-120:.8)   -- (-120:.2);
  \draw[thick] (-70:.91) -- (-110:.91);
  \begin{scope}[yshift = 1cm]
  \draw[thick] (60:.8)   -- (60:.2);
  \draw[thick] (120:.8)   -- (120:.2);
  \draw[thick] (70:.91) -- (110:.91);
  \end{scope}
  
  \begin{scope}[xshift = .5 cm, yshift = -0.86602540378 * 1 cm]
   \draw [fill = white] (-30:1) circle (2pt);
   \draw (-10:.6) node {\color{red}$\gamma_9$};
   \draw [thick] (-30:.2) -- (-30:.8);
  \end{scope}
  \begin{scope}[xshift = -.5 cm, yshift = -0.86602540378 * 1 cm]
   \draw [fill = white] (-150:1) circle (2pt);
   \draw (-170:.6) node {\color{red}$\gamma_8$};
   \draw [thick] (-150:.2) -- (-150:.8);
  \end{scope}
  \begin{scope}[xshift = .5 cm, yshift = 1.86602540378 * 1 cm]
   \draw [fill = white] (30:1) circle (2pt);
   \draw (5:.6) node {\color{red}$\gamma_8$};
   \draw [thick] (30:.2) -- (30:.8);
  \end{scope}
  \begin{scope}[xshift = -.5 cm, yshift = 1.86602540378 * 1 cm]
   \draw [fill = white] (150:1) circle (2pt);
   \draw [thick] (150:.2) -- (150:.8);
   \draw (170:.6) node {\color{red}$\gamma_9$};
  \end{scope}
   
  \end{scope}

\end{tikzpicture}

\caption{Fundamental domain of the $(3.12^2)$ tiling with edge weights.
In the monomeric case, all edge weights around triangles triangles are $\g_1 = \dots = \g_6 =: \alpha$ and the remaining weights are $\g_7 = \g_8 = \g_9 =: \beta$.
}
\label{fig:super_kagome}
\end{figure}
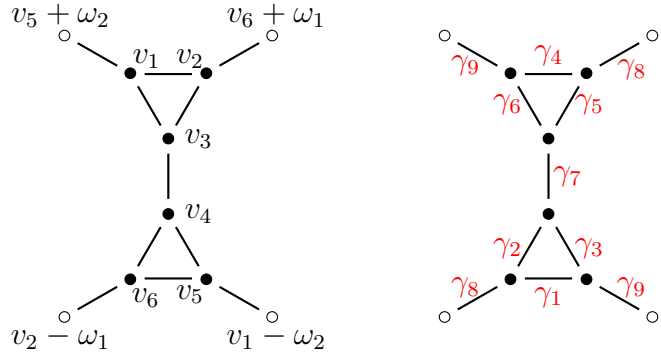
Given a constant vertex weight $\mu > 0$, the Floquet Laplacian~\eqref{eq:Laplacian_2} is a $6\times6$-matrix given by 
\begin{equation}\label{PerHamSK}
\Delta^{\theta}_{\gamma}=\Id
    - 
    \frac{1}{\mu}
    \pmat{
	0&\g_4&\g_6&0&z\g_9&0 \\ 
	\g_4&0&\g_5&0&0&w\g_8 \\ 
	\g_6&\g_5&0&\g_7&0&0 \\ 
	0&0&\g_7&0&\g_3&\g_2\\ 
	\overline{z}\g_9&0&0&\g_3&0&\g_1\\
	0&\overline{w}\g_8&0&\g_2&\g_1&0   
	},
\end{equation}
where $w:=e^{i\theta_1}$, $z:=e^{i\theta_2}$. 

\begin{itemize}
 \item 
If we fix a constant vertex weight $\mu > 0$, the condition $\sum_{w \sim v} \gamma_{vw} = \mu$ for all $v \in V$ leads to
\begin{equation}
    \label{eq:vertex_conditions_Super_Kagome}\begin{split}
    \mu 
    =
    \g_2+\g_3+\g_7
    =
    \g_5+\g_6+\g_7
    &=
    \g_1+\g_2+\g_8
    =
    \g_4+\g_5+\g_8 \\
    &=
    \g_1+\g_3+\g_9
    =
    \g_4+\g_6+\g_9.
    \end{split}
\end{equation}
This can be seen to be a linear system of $6$ linearly independent equations with $9$ unknowns, so the solution space is $3$-dimensional.
More precisely, by appropriate additions, we infer the three identities
\begin{equation} \label{ei}\begin{split}
    2 \g_1 + \g_8 + \g_9 &= 2 \g_7 + \g_2 + \g_3, \\
    2 \g_4 + \g_8 + \g_8 &= 2 \g_7 + \g_5 + d\g_6, \\
    \g_2 + \g_3 &= \g_5 + \g_6
\end{split}
\end{equation}
which imply $\gamma_1 = \gamma_4$.
The identities $\gamma_2 = \gamma_5$, and $\gamma_3 = \gamma_6$ follow by completely analogous calculations.
This leaves us with $6$ independent variables $\gamma_1, \gamma_2, \gamma_3$, and $\gamma_7, \gamma_8, \gamma_9$ which are however still subject to the three conditions
\[
  \g_2 + \g_3 + \g_7
  =
  \g_1 + \g_2 + \g_8
  =
  \g_1 + \g_3 + \g_9
  = 
  \mu
\]
from~\eqref{eq:vertex_conditions_Super_Kagome}.
Therefore, we are left with \emph{three} degrees of freedom.
    \item
If we additionally prescribe \emph{monomericity}, it is easy to see that there is only \emph{one} degree of freedom:
All edges around triangles carry the weight $\alpha > 0$, and all remaining edges (separating two dodecagons) carry the weight $\beta > 0$ under the condition $2 \alpha + \beta = \mu$.
\end{itemize}
\subsection{Flat bands in the perturbed Super-Kagome lattice}  
\begin{theorem}\label{thm:FlatBandsSuperKagome}
 Consider the perturbed Super-Kagome lattice with Laplacian~\eqref{eq:laplacian}, fixed vertex weight $\mu > 0$, and periodic edge weights $\gamma_1, \dots, \gamma_9 > 0$ satisfying the condition~\eqref{eq:vertex_and_edge_weights} on vertex and edge weights. Then, the following are equivalent:
 \begin{enumerate}[(i)]
  \item 
  There exist exactly two flat bands.
  \item
  The Super-Kagome lattice is monomeric. More explicitly, there are $\alpha,\beta > 0$ such that $ 2\alpha + \beta = \mu$ together with 
  \begin{equation*}\begin{split}
 \g_1=\g_2=\g_3=\g_4=\g_5=\g_6&=\alpha \ , \\
 \g_7=\g_8=\g_9&=\beta\ .
  \end{split}
 \end{equation*}
 \end{enumerate}
\end{theorem}

\begin{proof} Recall that in the constant vertex weight case, we have
\[ \g_1=\g_4\ , \quad \g_2=\g_5\ , \quad  \text{and} \quad \g_3=\g_6\ ,\] 
and consider the weighted adjacency matrix
\begin{equation}\Pi^{\theta}_{\gamma}:=\pmat{
	0&\g_4&\g_6&0&z\g_9&0 \\ 
	\g_4&0&\g_5&0&0&w\g_8 \\ 
	\g_6&\g_5&0&\g_7&0&0 \\ 
	0&0&\g_7&0&\g_3&\g_2\\ 
	\overline{z}\g_9&0&0&\g_3&0&\g_1\\
	0&\overline{w}\g_8&0&\g_2&\g_1&0}
=\pmat{
	0&\g_1&\g_3&0&z\g_9&0 \\ 
	\g_1&0&\g_2&0&0&w\g_8 \\ 
	\g_3&\g_2&0&\g_7&0&0 \\ 
	0&0&\g_7&0&\g_3&\g_2\\ 
	\overline{z}\g_9&0&0&\g_3&0&\g_1\\
	0&\overline{w}\g_8&0&\g_2&\g_1&0}
\end{equation}
which is a shifted and scaled version of $\Delta^{\theta}_{\gamma}$. 
%
We calculate
\begin{align*}
    &\det(\lambda\Id-\Pi^{\theta}_{\gamma})
    =\lambda^6
    -
    \lambda^4 
    \left( 2 \g_1^2 + 2 \g_2^2 + 2 \g_3^2 + \g_7^2 + \g_8^2 + \g_9^2 \right)-
    4
    \lambda^3 
    \g_1 \g_2 \g_3
    \\
    &+\lambda^2
        \big(
        \g_1^4 + \g_2^4 + \g_3^4
        +
        2 \g_1^2 \g_2^2 
        +
        2 \g_2^2 \g_3^2
        +
        2 \g_3^2 \g_1^2
        +
        2 \g_1^2 \g_7^2
        +
        2 \g_2^2 \g_9^2
        +
        2 \g_3^2 \g_8^2
        +
        \\
        &\qquad \qquad
        +
        \g_7^2 \g_8^2
        +
        \g_8^2 \g_9^2
        +
        \g_9^2 \g_7^2
        \big)
    \\
    &+
    4
    \lambda
        \g_1 \g_2 \g_3 
        \left(
            \g_1^2 + \g_2^2 + \g_3^2
        \right)
    \\
    &-
        \g_1^4 \g_7^2 
        -
        \g_2^4 \g_9^2
        -
        \g_3^4 \g_8^2
        -
        \g_7^2 \g_8^2 \g_9^2
        +
        4 \g_1^2 \g_2^2 \g_3^2
    \\
    &-
    \left( w + \overline{w} \right)
        \left(
            \lambda^2 \g_2^2 \g_7 \g_8 
            +
            2 \lambda \g_1 \g_2 \g_3 \g_7 \g_8
            +
            \g_1^2 \g_3^2 \g_7 \g_8 
            -
            \g_2^2 \g_7 \g_8 \g_9^2
        \right)
    \\
    &-
    \left( z + \overline{z} \right)
        \left(
            \lambda^2 \g_3^2 \g_7 \g_9 
            +
            2 \lambda \g_1 \g_2 \g_3 \g_7 \g_9
            +
            \g_1^2 \g_2^2 \g_7 \g_9
            -
            \g_3^2 \g_7 \g_8^2 \g_9
        \right)
    \\
    &-
    \left( w \overline{z} + \overline{w} z \right)
        \left(
            \lambda^2 \g_1^2 \g_8 \g_9 
            +
            2 \lambda \g_1 \g_2 \g_3 \g_8 \g_9
            +
            \g_2^2 \g_3^2 \g_8 \g_9
            -
            \g_1^2 \g_7^2 \g_8 \g_9
        \right).
\end{align*}
Since $w + \overline w = 2 \cos(\theta_1)$, $z + \overline z = 2 \cos(\theta_2)$, and $w \overline z + \overline w z = 2 \cos (\theta_1 - \theta_2)$ are linearly on $\TT^2$, $\lambda$ is a $\theta$-independent eigenvalue if and only if the conditions

\begin{equation}
    \label{CondIXXX}
 \begin{split}
 \lambda^2 \g_2^2
            +
            2 \lambda \g_1 \g_2 \g_3 
            +
            \g_1^2 \g_3^2 
            -
            \g_2^2 \g_9^2
            &=
            0,
            \\
  \lambda^2 \g_3^2
            +
            2 \lambda \g_1 \g_2 \g_3 
            +
            \g_1^2 \g_2^2 
            -
            \g_3^2 \g_8^2 
            &=
            0,
            \\
  \lambda^2 \g_1^2
            +
            2 \lambda \g_1 \g_2 \g_3 
            +
            \g_2^2 \g_3^2 
            -
            \g_1^2 \g_7^2 
            &= 0
            ,
\end{split}
\end{equation}
as well as
\begin{equation}
    \label{KugelbedingungSK}
   \begin{split}
 &\lambda^6
    -
    \lambda^4
        \left( 2 \g_1^2 + 2 \g_2^2 + 2 \g_3^2 + \g_7^2 + \g_8^2 + \g_9^2 \right)
        -
    4
    \lambda^3 
        \g_1 \g_2 \g_3
   \\
    &+\lambda^2
        \big(
        \g_1^4 + \g_2^4 + \g_3^4
        +
        2 \g_1^2 \g_2^2 
        +
        2 \g_2^2 \g_3^2
        +
        2 \g_3^2 \g_1^2
        +
        2 \g_1^2 \g_7^2
        +
        2 \g_2^2 \g_9^2
        +
        2 \g_3^2 \g_8^2
        +
        \\
        &\qquad \qquad
        +
        \g_7^2 \g_8^2
        +
        \g_8^2 \g_9^2
        +
        \g_9^2 \g_7^2
        \big)
    \\
    +&
    4
    \lambda
        \g_1 \g_2 \g_3 
        \left(
            \g_1^2 + \g_2^2 + \g_3^2
        \right)
    \\
    -&
        \g_1^4 \g_7^2 
        -
        \g_2^4 \g_9^2
        -
        \g_3^4 \g_8^2
        -
        \g_7^2 \g_8^2 \g_9^2
        +
        4 \g_1^2 \g_2^2 \g_3^2=0
    \end{split}
\end{equation}
hold.\footnote{As we will see later, despite its complexity, \eqref{KugelbedingungSK} will not impose further restrictions and hold in all relevant cases. This appears to be a consequence of symmetries of the lattice and the operator.}
	Conditions~\eqref{CondIXXX} imply that any $\theta$-independent eigenvalue of the matrix $\Pi^{\theta}_{\gamma}$ must satisfy
	\[
	 \lambda
	 =
	 - \frac{\g_1 \g_3}{\g_2} \pm \g_9,
	 \quad
	 \lambda
	 =
	 - \frac{\g_1 \g_2}{\g_3} \pm \g_8,
	 \quad
	 \text{and}
	 \quad
	 \lambda
	 =
	 - \frac{\g_2 \g_3}{\g_1} \pm \g_7.
	\]
Since all $\gamma_i$ are positive, the only way for these three equations to have the same set of solutions, that is for two flat bands to exist, is therefore

	\begin{equation}\label{HHH}
	-\frac{\g_1\g_3}{\g_2}+\g_9=-\frac{\g_1\g_2}{\g_3}+\g_8=-\frac{\g_2\g_3}{\g_1}+\g_7
	\end{equation}  
	together with
	\begin{equation}\label{HHH2}
	-\frac{\g_1\g_3}{\g_2}-\g_9=-\frac{\g_1\g_2}{\g_3}-\g_8=-\frac{\g_2\g_3}{\g_1}-\g_7.
	\end{equation}  
	This implies that the matrix $\Pi_\gamma^\theta$ can only have two $\theta$-independent eigenvalues if there are $\alpha, \beta > 0$ with
	\[ 
	\alpha = \g_7=\g_8=\g_9 \qquad \mbox{ and }\qquad \beta = \g_1=\g_2=\g_3 , 
	\]
	that is the monomeric case, and the only candidates for these eigenvalues are $-\beta \pm \alpha$.
	To see that they are indeed eigenvalues, one verifies by an explicit calculation that condition~\eqref{KugelbedingungSK} is also fulfilled.
	This shows the stated equivalence.

\end{proof}

	Next, we further describe the spectrum of the monomeric Super-Kagome lattice.
	\begin{theorem}[Band gaps in the perturbed Super-Kagome lattice]\label{thm:SUPERKagomeGaps}
    Consider the perturbed Super-Kagome lattice with Laplacian~\eqref{eq:laplacian} with fixed vertex weight $\mu > 0$ and monomeric edge weights $\alpha, \beta > 0$, satisfying $2 \alpha + \beta = \mu$ as characterized in Theorem~\ref{thm:FlatBandsSuperKagome}.
    Then, the spectrum is given by
    \[
     I_1 \cup I_2
     :=
    \left[
     0,
     \left(
        1 - \frac{\alpha}{2 \mu} 
     \right)
     - 
     \frac{\lvert 3 \alpha - 2 \beta \rvert}{2 \mu}
     \right]
     \bigcup
     \left[
     \left(
        1 - \frac{\alpha}{2 \mu} 
     \right)
     + 
     \frac{\lvert 3 \alpha - 2 \beta \rvert}{2 \mu}
     ,
     2 - \frac{\alpha}{\mu}
     \right]  
     \]
     with flat bands at $\frac{3 \alpha}{\mu}$ and $2 - \frac{\alpha}{\mu}$.
	\end{theorem}

     The spectrum and the position of the flat bands have been plotted in Figure~\ref{fig:Super_Kagome}.
     The spectrum generically consists of two distinct intervals (bands) except for the case $3 \alpha = 2 \beta$, that is $\alpha = \frac{2 \mu}{7}$, in which the two bands touch and the spectrum consists of one interval with an embedded flat band in the middle as well as a flat band at its maximum.
     This case $\alpha = \frac{2 \mu}{7}$ connects two regimes with different spectral pictures:
     \begin{itemize}
      \item      
     If $\alpha > \frac{2 \mu}{7}$ the spectrum consists of two intervals the upper one of which has \emph{two flat bands at its endpoints}.
     In the special case of uniform edge weights (that is $\alpha = \frac{\mu}{3}$, this has already been observed, for instance in~\cite{TaeuferPeyerimhoff}.
     \item
     If $\alpha < \frac{2 \mu}{7}$, the spectrum will again consist of two intervals each of which will have a flat band at its maximum.
     Somewhat surprisingly, the lower flat band has now attach itself to the lower interval $I_2$ upon passing the critical parameter $\alpha = \frac{2 \mu}{7}$.
     \end{itemize}
     Another noteworthy observation is that no gap opens within the intervals $I_1$ and $I_2$, despite them being generated by two distinct Floquet eigenvalues and the density of states measure vanishing at a point in the interior of the bands, see again~\cite{TaeuferPeyerimhoff} for plots of the integrated density of states in the case of constant edge weights.
     In particular, this distinguishes the monomeric Super-Kagome lattice from the monomeric Kagome lattice where such a gap indeed opens within the spectrum at points of zero spectral density.

	\begin{figure}[ht]
    \begin{tikzpicture}[xscale = 4, yscale = 8]
        
        
        \fill[red!20] (0,0) -- (3 * .28571428571, .28571428571) -- (0,.5) -- (0,0);
        
        \fill [red!20] (2,0) -- (3 * .28571428571, .28571428571) -- (1.5, .5) -- (2,0);

        \draw[red, very thick, dashed] (2,0) -- (1.5,.5);
        \draw[red, very thick, dashed] (0,0) -- (1.5,.5);
        
        \draw (.35,.25) node {$I_1$};
        \draw (1.45,.25) node {$I_2$};
        
        \draw[very thick, dotted] (0,.333333) -- (2,.333333);

        \draw[thick, ->] (0,0) -- (2.05,0);
        \draw[thick] (0,0) -- (0,.5);
        
        \draw (-.2,.5) node {$\alpha = \frac{\mu}{2}$};
        \draw[thick] (-.05,.5) -- (.05,.5);
        \draw (-.2,.333) node {$\frac{\mu}{3}$};
        \draw[thick] (-.05,.333333) -- (.05,.333333);

        \draw (-.1,.28571428571) node {$\frac{2 \mu}{7}$};
        \draw[thick] (-.05,.28571428571) -- (.05,.28571428571);        
        \draw (-.2,0) node {$\alpha = 0$};
        
        
        \draw (2,-.1) node {$2$};
        \draw[thick] (2,.025) -- (2,-.025);

        \begin{scope}[xshift = 2.2cm]
         \draw[very thick, red, dashed] (0,0) -- (.25,0);
         \draw[anchor = west] (.25,0) node {Flat bands};

         \fill[red!20] (0,.1) rectangle (.25,.2);
         \draw[anchor = west] (.25,.15) node {$\sigma(\Delta_\gamma)$};
         
         \draw[very thick, dotted] (0,.3) -- (.25,.3);
         \draw[anchor = west] (.25,.3) node {Constant edge weights};
         

        \end{scope}
    \end{tikzpicture}

    \caption{Spectrum of the monomeric $(3.12^2)$ ``Super-Kagome'' lattice with vertex weight $\mu > 0$ as a function of the parameter $\alpha \in (0, \frac{\mu}{2})$, describing the edge weights on edges adjacent to triangles.}
    \label{fig:Super_Kagome}
	\end{figure}
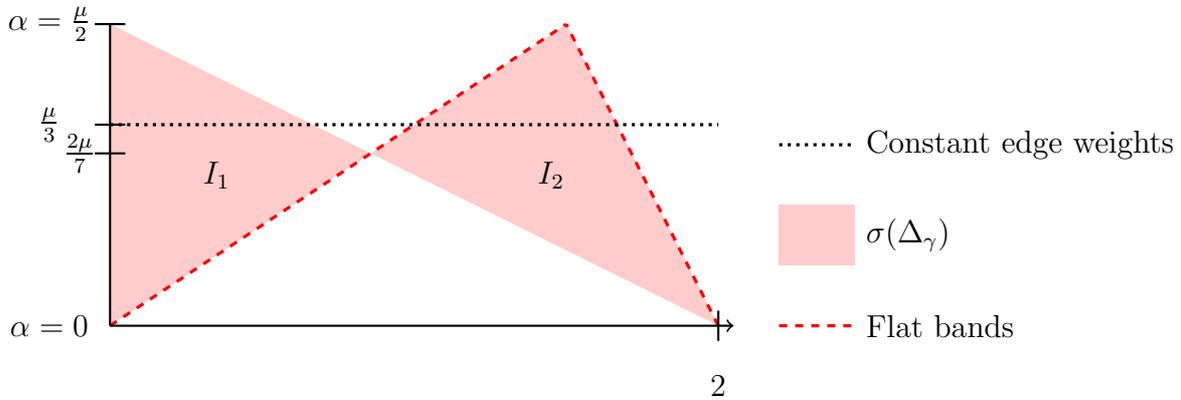

	\begin{proof}[Proof of Theorem~\ref{thm:SUPERKagomeGaps}]
	 In the monomeric case, the characteristic polynomial $ \det (\lambda \Id - \Pi_\gamma^\theta)$ of the matrix
	 $\Pi_\gamma^\theta$ simplifies to
	 \begin{equation*}\begin{split}
	 (
	 (&\alpha + \lambda)^2 - \beta^2
	 )\cdot \\
	 &(
	 \lambda^4
	 - 
	 2 \alpha \lambda^3
	 -
	 (3 \alpha^2 + 2 \beta^2) \lambda^2
	 + 
	 (4 \alpha^3 + 2 \alpha \beta^2) \lambda
	 +
	 4 \alpha^4+ \alpha^2 \beta^2 + \beta^4 - 2 \alpha^2 \beta^2 F(\theta_1,\theta_2))\ ,
	 \end{split}
	 \end{equation*}
	 where $F(\theta_1, \theta_2) = \cos(\theta_1) + \cos(\theta_2) + \cos(\theta_1 + \theta_2)$. Its six roots are
    $$
     \left\{
     -\alpha \pm \beta,
     \frac{1}{2}
     \left(
        \alpha 
        \pm
        \sqrt{9 \alpha^2 + 4 \beta^2
        \pm
        4 \alpha \beta \sqrt{3 + 2 F(\theta_1,\theta_2)}
        }
     \right)
     \right\}\ ,
    $$
    whence the eigenvalues of $\Delta_\gamma^\theta$ are given by
    \begin{align*}
    \lambda_1(\theta, \gamma)
    &=
    1 
    -
    \frac{1}{2 \mu}
    \left(
        \alpha 
        +
        \sqrt{9 \alpha^2 + 4 \beta^2 
        +
        4 \alpha \beta \sqrt{3 + 2 F(\theta_1,\theta_2)}
        }
    \right)
    \ ,
    \\
    \lambda_2(\theta, \gamma)
    &=
    1 
    -
    \frac{1}{2 \mu}
    \left(
        \alpha
        +
        \sqrt{9 \alpha^2 + 4 \beta^2 
        -
        4 \alpha \beta \sqrt{3 + 2 F(\theta_1,\theta_2)}
        }
    \right)\ ,
    \\
    \lambda_3(\theta, \gamma)
    &=
    1 + \frac{\alpha - \beta}{\mu}
    =
    \frac{3 \alpha}{\mu}
    =
    \begin{cases}
    1 - \frac{\alpha - \lvert 3 \alpha - 2 \beta \rvert}{2 \mu} 
    &\text{if $3 \alpha \geq 2 \beta$}\ ,
    \\
    1 - \frac{\alpha - \lvert 3 \alpha - 2 \beta \rvert}{2 \mu}
    & \text{if $3 \alpha < 2 \beta$}\ ,
    \end{cases}    
    \end{align*}
    \begin{align*}
    \lambda_4(\theta, \gamma)
    &=
    1 
    -
    \frac{1}{2 \mu}
     \left(
        \alpha 
        -
        \sqrt{9 \alpha^2 + 4 \beta^2 
        -
        4 \alpha \beta \sqrt{3 + 2 F(\theta_1,\theta_2)}
        }
    \right)\ ,
    \\
    \lambda_5(\theta, \gamma)
    &=
    1 
    -
    \frac{1}{2 \mu}
     \left(
        \beta 
        -
        \sqrt{9 \alpha^2 + 4 \beta^2 
        +
        4 \alpha \beta \sqrt{3 + 2 F(\theta_1,\theta_2)}
        }
    \right)\ ,
    \\
    \lambda_6(\theta, \gamma)
    &=
    1 + \frac{ \alpha + \beta }{\mu}
    =
    2 - \frac{\alpha}{\mu}\ .
    \end{align*}
    Using that the map $\TT^2 \ni (\theta_1, \theta_2) \mapsto F(\theta_1, \theta_2)$ takes all values in the interval $(- 3/2, 3)$, we conclude that the bands, generated by $\lambda_1(\theta, \gamma)$ and $\lambda_2(\theta, \gamma)$, as well as the bands generated by $\lambda_4(\theta, \gamma)$ and $\lambda_5(\theta, \gamma)$ always touch, and the spectrum consists of the two intervals
    \begin{align*}
     &\left[
     \min_{\theta \in \TT^2}
     \lambda_1(\theta, \gamma)
     ,
     \max_{\theta \in \TT^2}
     \lambda_2(\theta, \gamma)
     \right]
     \bigcup
     \left[
     \min_{\theta \in \TT^2}
     \lambda_4(\theta, \gamma)
     ,
     \max_{\theta \in \TT^2}
     \lambda_5(\theta, \gamma)
     \right]
     \\
     =
     &\left[
     0,
     1 - \frac{\alpha + \lvert 3 \alpha - 2 \beta \rvert}{2 \mu}
     \right]
     \bigcup
     \left[
     1 - \frac{\alpha - \lvert 3 \alpha - 2 \beta \rvert}{2 \mu}
     ,
     2 - \frac{\alpha}{2 \mu}
     \right]
     \\
     =
     &\left[
     0,
     \left(
        1 - \frac{\alpha}{2 \mu} 
     \right)
     - 
     \frac{\lvert 3 \alpha - 2 \beta \rvert}{2 \mu}
     \right]
     \bigcup
     \left[
     \left(
        1 - \frac{\alpha}{2 \mu} 
     \right)
     + 
     \frac{\lvert 3 \alpha - 2 \beta \rvert}{2 \mu}
     ,
     2 - \frac{\alpha}{\mu}
     \right].
     \qedhere
    \end{align*}

	\end{proof}

One might now wonder under which conditions \emph{only one flat band} exists.
The next theorem completely identifies all parameters for which one flat band exists:

\begin{theorem}
    \label{thm:one_flat_band_SK}
 Consider the perturbed Super-Kagome lattice with Laplacian~\eqref{eq:laplacian}, fixed vertex weight $\mu > 0$, and periodic edge weights $\gamma_1, \dots, \gamma_9 > 0$ satisfying the condition \eqref{eq:vertex_and_edge_weights} on vertex and edge weights.
 The set of $(\gamma_i)$ such that exactly one flat band exists consists of \textbf{six connected components} which have no mutual intersections and have no intersection with the two-flat-band parameter set, identified in Theorem~\ref{thm:FlatBandsSuperKagome}.
 
 The solution space is invariant under those permutations of the $\gamma_i$ which correspond to rotations of the lattice by $\frac{2\pi}{3}$, and $\frac{4\pi}{3}$.
 Modulo these permutations, the two connected components can be described as follows
 \begin{itemize}
  \item 
  A one-dimensional submanifold, isomorphic to an interval, and explicitely descibed in equation~\eqref{eq:solution_Case-++},
  \item
  Two one-dimensional submanifolds each isomorphic to an interval, explicitely described in~\eqref{eq:solution_Case+--_a}, and~\eqref{eq:solution_Case+--_b}, which intersect in a single point.
  \end{itemize}
\end{theorem}

\begin{proof}[{Proof of Theorem~\ref{thm:one_flat_band_SK}}]
    Recall that due to the reductions made at the beginning of the section, after fixing the constant vertex weight $\mu > 0$, the space of edge weights is a $3$-dimensional manifold in the $6$-dimensional parameter space $\{ \gamma_1, \gamma_2, \gamma_3, \gamma_7, \gamma_8, \gamma_9 > 0 \}$, subject to the conditions
    \begin{equation}
    \label{eq:vertex_weights_proof}
     \g_1+\g_3+\g_9= \g_1+\g_2+\g_8=\g_2+\g_3+\g_7=\mu .
    \end{equation}
    Furthermore, from the proof of Theorem~\ref{thm:FlatBandsSuperKagome} we infer that $\Delta_\gamma$ has a flat band at $\lambda$ if and only if
    the weighted adjacency matrix $\Pi_\gamma^\theta$ has the $\theta$-independent eigenvalue $\tilde \lambda := \mu(1-\lambda)$.
    This requires in particular that
    \begin{equation}\label{EquationPROOFXXX}
	\tilde \lambda 
	=
	-\frac{\g_1\g_3}{\g_2}\pm\g_9=-\frac{\g_1\g_2}{\g_3}\pm\g_8=-\frac{\g_2\g_3}{\g_1}\pm\g_7
	\end{equation} 
    holds with a certain combination of plus and minus signs.
    Now, if equality in~\eqref{EquationPROOFXXX} holds with all three signs positive or all three signs negative, respectively, then the argument in the proof of Theorem~\ref{thm:FlatBandsSuperKagome} shows that this already implies that the edge weights are monomeric, the identities also hold with the opposite sign, the additional condition~\eqref{KugelbedingungSK} is fulfilled, and there are two flat bands.
    As a consequence, the only chance for the existence of exactly one flat band is~\eqref{EquationPROOFXXX} to hold with different signs in front of $ \gamma_7, \gamma_8, \gamma_9 $. Also, it is immediately clear that 
  \eqref{EquationPROOFXXX} with different signs does not allow for a monomeric and non-zero solution and hence the solution space consists of at most six mutually disjoint components which have no intersection with the two-flat-band manifold, identified in Theorem~\ref{thm:FlatBandsSuperKagome}.
    

    By symmetry, it suffices to investigate two out of these six cases:
    \begin{equation}
    \label{eq:case_1}
    \text{\textbf{Case(-\ +\ +):}}
    \qquad
     -\frac{\g_1\g_3}{\g_2} - \g_9=-\frac{\g_1\g_2}{\g_3}+\g_8=-\frac{\g_2\g_3}{\g_1}+\g_7
     =
     \tilde \lambda\ ,
    \end{equation}
    and
    \begin{equation}
    \label{eq:case_2}
    \text{\textbf{Case(+\ -\ -):}}
    \qquad
     -\frac{\g_1\g_3}{\g_2} + \g_9=-\frac{\g_1\g_2}{\g_3}-\g_8=-\frac{\g_2\g_3}{\g_1}-\g_7
    =
     \tilde \lambda\ .
    \end{equation}
    To solve \textbf{Case(-\ +\ +)}, combine the second identities in in~\eqref{eq:vertex_weights_proof} and~\eqref{eq:case_1}, to deduce
    \[
     \gamma_3 - \gamma_1
     =
     \frac{\gamma_2}{\gamma_1 \gamma_3} (\gamma_1^2 - \gamma_3^2)
    \]
    which, recalling $\gamma_i > 0$, is only possible if $\gamma_1 = \gamma_3$. 
    But then, by~\eqref{eq:case_1}, $\gamma_7 = \gamma_8$.
    Calling $\alpha' := \gamma_2$, and $\beta' := \gamma_9$, we can use~\eqref{eq:vertex_weights_proof}, to further express
    \begin{equation}
    \label{eq:alpha_and_beta_used_case_gamma_1=gamma_3}
     \gamma_1 = \gamma_3 = \frac{\mu - \beta'}{2},
     \quad
     \text{and}
     \quad
     \gamma_7 = \gamma_8 = \frac{\mu + \beta'}{2} - \alpha'.
    \end{equation}
    Next, we eliminate $\beta'$ by resolving the yet unused first identity in~\eqref{eq:case_1}, which yields
    \begin{align*}
    &- \frac{( \mu - \beta')^2}{4 \alpha'} - \beta'
    =
    - \alpha' + \frac{\mu + \beta'}{2} - \alpha'
    \\
    \Leftrightarrow
    \quad
    &
    \beta'
    =
    \mu -  3 \alpha' 
    \pm
    \sqrt{17 \alpha'^2 - 8 \alpha' \mu}.
    \end{align*}
    This only has real solutions if $\alpha' > \frac{8}{17} \mu > \frac{1}{3} \mu$, thus only
    \[
    \beta'
    =
    \mu -  3 \alpha' 
    +
    \sqrt{17 \alpha'^2 - 8 \alpha' \mu}.
    \]
    can be a positive solution.
    Furthermore, we need $\beta' \in (0, \mu)$, which is the case if and only if 
    \[
    \gamma_2 = \alpha' \in \left( \frac{\mu}{2}, \mu \right).
    \]
    We therefore find the one-parameter solution set
    \begin{equation}
    \label{eq:solution_Case-++}
    \text{\textbf{Case (-\ +\ +)}}
    \quad 
    \begin{cases}
    	\g_1 = \g_3 
		&= \frac{\mu - \beta'}{2},\\
		\g_2 = \alpha' 
		&\in \left( \frac{\mu}{2}, \mu \right),\\
		\g_7 = \g_8 
		&= \frac{\mu + \beta'}{2} - \alpha',\\
		\g_9 
		= \beta' 
		&:= \mu - 3 \alpha' + \sqrt{17 \alpha'^2 - 8 \alpha \mu}
    \end{cases}
    \end{equation}
    with energy
    \[
     \tilde \lambda
     =
     - \g_2 + \g_7
     =
     -2 \alpha' + \frac{\mu + \beta'}{2}
     =
     -2 \alpha'
     +
     \frac{2 \mu - 3 \alpha' + \sqrt{17 \alpha'^2-8 \alpha' \mu}}{2}.
    \]
    Finally, an explicit calculation shows that with these parameters,~\eqref{KugelbedingungSK} is indeed fulfilled.
    
    As for \textbf{Case(+\ -\ -)}, we combine the second identity in~\eqref{eq:vertex_weights_proof} with the second identity in~\eqref{eq:case_2} to deduce
    \begin{equation}
    \label{eq:case_+--}
     \gamma_3 - \gamma_1
     =
     \frac{\gamma_2}{\gamma_1 \gamma_3} (\gamma_3^2 - \gamma_1^2)\ .
    \end{equation}
    Identity~\eqref{eq:case_+--} has two types of solutions:
    \\
    \textbf{Case(+\ -\ -)(a)}: 
    $\gamma_1 = \gamma_3$.\\
    As before we find $\gamma_7 = \gamma_8$. Let $\alpha' := \gamma_2$, $\beta' := \gamma_9$, and combine the remaining first identity in~\eqref{eq:case_2} with~\eqref{eq:alpha_and_beta_used_case_gamma_1=gamma_3} to solve for $\beta'$, finding
    \begin{align*}
    &- \frac{(\mu - \beta')^2}{4 \alpha'}
    +
    \beta'
    =
    - 
    \frac{\mu + \beta'}{2}
    \\
    \Leftrightarrow
    \quad
    &
    \beta'
    =
    \mu + 3 \alpha'
    \pm
    \sqrt{9 \alpha'^2 + 8 \alpha' \mu}.
    \end{align*}
    Only the solution 
    \[
    \beta'
    =
    \mu + 3 \alpha'
    -
    \sqrt{9 \alpha'^2 + 8 \alpha' \mu}
    \]
    has a chance to be in $(0, \mu)$, and, indeed, this is the case if and only if
    \[
    \gamma_2 = \alpha' \in \left(0, \frac{\mu}{2} \right).
    \]
    We obtain the one-parameter solution set
    \begin{equation}
    \label{eq:solution_Case+--_a}
    \text{\textbf{Case(+\ -\ -)(a)}}
    \quad 
    \begin{cases}
    	\g_1 = \g_3 
		&= \frac{\mu - \beta'}{2},\\
		\g_2 = \alpha' 
		&\in \left(0, \frac{\mu}{2} \right),\\
		\g_7 = \g_8 
		&= \frac{\mu + \beta'}{2} - \alpha',\\
		\g_9 
		= \beta' 
		&:= \mu + 3 \alpha' - \sqrt{9 \alpha'^2 + 8 \alpha' \mu}
    \end{cases}
    \end{equation}
    with energy
    \[
     \tilde{\lambda}
     =
     - \g_2 - \g_7
     =
     - \frac{\mu + \beta'}{2}
     =
     -
     \frac{2 \mu + 3 \alpha' - \sqrt{9 \alpha'^2+8 \alpha' \mu}}{2}.
    \]
    Again, an explicit calculation shows that~\eqref{KugelbedingungSK} is fullfilled.
    
    \textbf{Case(+\ -\ -)(b)}: The other solution of~\eqref{eq:case_+--} is 
    \[
    \gamma_1 \gamma_3 = \gamma_2 (\gamma_1 + \gamma_3).
    \]
    We set $\alpha'' := \gamma_1$, $\beta'' := \gamma_3$, whence
	\[
    \gamma_2
    =
    \frac{\alpha'' \beta''}{\alpha'' + \beta''},
    \]
    and use~\eqref{eq:vertex_weights_proof} to infer
    \begin{equation}
    \label{eq:gamma_789}
    \gamma_7
    =
    \mu
    -
    \frac{2 \alpha'' \beta'' + \beta''^2}{\alpha'' + \beta''}
    ,
    \quad
    \gamma_8
    =
    \mu - \frac{\alpha''^2 + 2 \alpha'' \beta''}{\alpha'' + \beta''},    
    \quad
    \gamma_9
    =
    \mu - \alpha'' - \beta''.
    \end{equation}
	Plugging~\eqref{eq:gamma_789} into the yet unused first identity in~\eqref{eq:case_2}, we arrive at
	\begin{align*}
	&- (\alpha'' + \beta'')
	+
	\mu - \alpha'' - \beta''
	=
	- \frac{\alpha''^2}{\alpha'' + \beta''}
	-
	\mu
	+
	\frac{\alpha''^2 + 2 \alpha'' \beta''}{\alpha'' + \beta''}
	\\
	\Leftrightarrow
	\quad
	&
	\beta''
	=
	\frac{\mu - 3 \alpha'' \pm \sqrt{(\mu - 3 \alpha'')^2 + 4 \alpha''(\mu - \alpha'')}}{2}
	=
	\frac{\mu - 3 \alpha'' \pm \sqrt{\mu^2 - 2 \alpha'' \mu + 5 \alpha''^2}}{2}
	\end{align*}
	We observe that only the solution with a plus has a chance to be positive and it is easy to see that this solution takes values in $(0, \mu)$ for all $\alpha'' \in (0, \mu)$.
	We obtain the one-parameter solution set
	\begin{equation}
    \label{eq:solution_Case+--_b}
    \text{\textbf{Case (+\ -\ -) (b)}}
    \quad 
    \begin{cases}
    	\g_1 
		= \alpha''
		&\in \left(0 ,  \mu \right),\\
		\g_2
		&=
		\frac{\alpha'' \beta''}{\alpha'' + \beta''},\\
		\g_3 = \beta''
		&:=
		\frac{\mu - 3 \alpha'' + \sqrt{\mu^2 - 2 \alpha'' \mu + 5 \alpha''^2}}{2},\\
		\g_7
    	&=
    	\mu
    	-
    	\frac{2 \alpha'' \beta'' + \beta''^2}{\alpha'' + \beta''},\\
    	\g_8
    	&=
    	\mu - \frac{\alpha''^2 + 2 \alpha'' \beta''}{\alpha'' + \beta''},\\    
 	  	\g_9
		&=
    	\mu - \alpha'' - \beta''
    \end{cases}
    \end{equation}
    at energy
    \[
     \tilde{\lambda}
     =
     - \frac{\g_1 \g_3}{\g_2}
     +
     \g_9
     =
     \mu - 2 \alpha'' - 2 \beta''
     =
     \alpha - \sqrt{\mu^2 - 2 \alpha'' \mu + 5 \alpha''^2}.
    \]
    Again, an explicit calculation verifies that with these choices, \eqref{KugelbedingungSK} is fullfilled.

    Finally, to conclude the claimed topological properties of the manifolds, we need to verify that the solution space~\eqref{eq:solution_Case+--_a} in \textbf{Case(+\ -\ -)(a)} intersects the solution space~\eqref{eq:solution_Case+--_b} in \textbf{Case(+\ -\ -)(b)} if and only if
    \[
     \g_1 = \g_3 = \g_7 = \g_8 = \frac{2 \mu}{5},
     \quad
     \g_2 = \g_9 = \frac{\mu}{5}.
     \qedhere
    \]   
\end{proof}

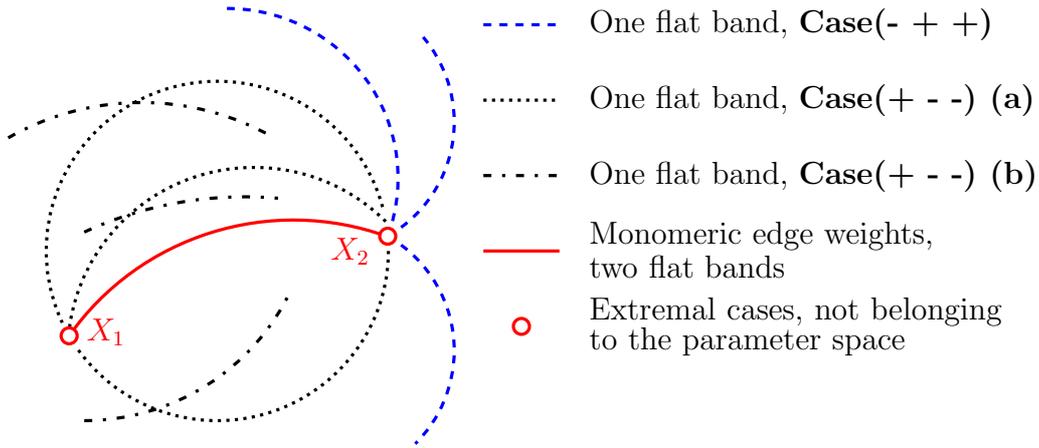
\begin{figure}[ht]
    \begin{tikzpicture}[rotate = 0, scale=1]
     
      \draw[very thick, dotted] (15:2.25) arc (15:205:2.25);
      \draw[very thick, dotted] (0:2.25) arc (360:215:2.25);
      \draw[very thick, dotted] (10:2.25) arc (45:172:2.45);
      \draw[very thick, loosely dashdotted] (-2.75,1.5) arc (120:60:3.5);
      \draw[very thick, loosely dashdotted] (-1.75,.25) arc (115:85:5);      
      \draw[very thick, loosely dashdotted] (-1.75,-2.25) arc (270:335:3);
      

      \draw[very thick, dashed, blue] (5:2.25) arc (-20:90:2.25);
      \draw[very thick, dashed, blue] (5:2.25) arc (-60:40:1.75);
      \draw[very thick, dashed, blue] (5:2.25) arc (60:-45:1.75);
      
      \draw[red, very thick] (210:2.25) arc (145:70:3.6);
      \draw[very thick, red, fill= white] (210:2.25) circle (3pt); 
      \draw[very thick, red, fill = white] (5:2.25) circle (3pt);
        
      \draw[anchor = east] (0:2.15) node {\color{red}$X_2$};
      \draw[anchor = west] (210:2.15) node {\color{red}$X_1$};      
    
     \begin{scope}[xshift = 3.5cm]
     \draw[very thick, dashed, blue] (0,3) -- (1,3);
     \draw[anchor = west] (1.25,3) node {One flat band, \textbf{Case(-\ +\ +)}};
     \draw[very thick, dotted] (0,2) -- (1,2);
     \draw[anchor = west] (1.25,2) node {One flat band, \textbf{Case(+\ -\ -) (a)}};
     \draw[very thick, loosely dashdotted] (0,1) -- (1,1);
     \draw[anchor = west] (1.25,1) node {One flat band, \textbf{Case(+\ -\ -) (b)}};
     \draw[red, very thick] (0,0) -- (1,0);
     \draw[anchor = west] (1.25,.2) node {Monomeric edge weights,} ;
     \draw[anchor = west] (1.25,-.2) node {two flat bands};
     \draw[red, very thick] (0.5,-1) circle (3pt);
     \draw[anchor = west] (1.25,-.8) node {Extremal cases, not belonging};
     \draw[anchor = west] (1.25,-1.2) node {to the parameter space};
    \end{scope}

 \end{tikzpicture}

 \caption{Schematic overview of the topology of the six ``spurious'' one-flat-band solution sets, and the monomeric two-flat-band manifold within the constant-vertex weight parameter space.
 \textbf{Case(-\ +\ +)} solutions asymptotically meet the limit points of the two-flat-band manifold at one end of the parameter range, whereas \textbf{Case(+\ -\ -) (a)} solutions asymptotically meet it at both ends of the parameter range.}
 \label{fig:topology_one_flat_band_solutions}
 
\end{figure}

\begin{remark}
Theorems~\ref{thm:FlatBandsSuperKagome} and~\ref{thm:one_flat_band_SK} imply that the six one-flat-band components and the two-flat-band component are mutually disjoint.
However, a closer analysis of the extremal cases in Formulas~\eqref{eq:solution_Case-++}, \eqref{eq:solution_Case+--_a}, and~\eqref{eq:solution_Case+--_b}, as well as of the monomeric case, implies that when sending the parameters to their extremal values, the three one-dimensional manifolds corresponding to \textbf{Case(+\ -\ -) (a)}, and the two-flat-band-manifold of solutions converge to the two points
\begin{align*}
X_1 := \left(0,0,0,\frac{\mu}{2},\frac{\mu}{2},\frac{\mu}{2} \right)
\quad
\text{and}
\quad
X_2
:=
\left(\frac{\mu}{2},\frac{\mu}{2},\frac{\mu}{2},0,0,0 \right),
\end{align*}
which themselves do no longer belong to the space of admissible parameters.
Likewise, the limit of solutions of \textbf{Case(+\ -\ -)} in~\eqref{eq:solution_Case-++} corresponding to $\alpha' = \frac{\mu}{2}$ corresponds to the the point $X_2$, see also Figure~\ref{fig:topology_one_flat_band_solutions}.
\end{remark}

	\subsection*{Acknowledgement}{JK would like to thank the Bergische Universität Wuppertal where parts of this project were done while being on leave from the FernUniversität in Hagen. JK and MT also acknowledge support by the Cost action CA18232 through the summer school ``Heat Kernels and Geometry: From Manifolds to Graphs'' held in Bregenz.
	MT would like to thank the Mittag-Leffler Institute where parts of this work were initiated during the trimester Program ``Spectral Methods in Mathematical Physics''.}

\vspace*{0.5cm}


\newcommand{\etalchar}[1]{$^{#1}$}
\def\cprime{$'$} \def\polhk#1{\setbox0=\hbox{#1}{\ooalign{\hidewidth
  \lower1.5ex\hbox{`}\hidewidth\crcr\unhbox0}}}
\providecommand{\bysame}{\leavevmode\hbox to3em{\hrulefill}\thinspace}
\providecommand{\MR}{\relax\ifhmode\unskip\space\fi MR }
\providecommand{\MRhref}[2]{%
  \href{http://www.ams.org/mathscinet-getitem?mr=#1}{#2}
}
\providecommand{\href}[2]{#2}

\end{document}